\documentclass[11pt]{amsart}
\usepackage{amsmath}
\usepackage{amsfonts}
\usepackage{amsthm}
\usepackage{amssymb}

\theoremstyle{plain}

\numberwithin{equation}{section}

\newtheorem{Theorem}{Theorem}
\newtheorem{Lemma}[Theorem]{Lemma}

\newtheorem{Proposition}[Theorem]{Proposition}

\newtheorem{Corollary}[Theorem]{Corollary}

\newtheorem{Example}[Theorem]{Example}

 \theoremstyle{remark}

\title[Dirac operators with special periodic
potentials] {1D Dirac operators with special periodic potentials}

\author{Plamen Djakov}

\author{Boris Mityagin}

\begin{document}

\begin{abstract}
For one-dimensional Dirac operators of the form
$$ Ly= i \begin{pmatrix}  1  &  0 \\ 0  & -1 \end{pmatrix}
\frac{dy}{dx} + v y, \quad v= \begin{pmatrix}  0  &  P \\ Q  & 0
\end{pmatrix}, \;\; y=\begin{pmatrix}  y_1 \\ y_2 \end{pmatrix}
$$
we single out a class $X$ of $\pi$-periodic potentials $v$ with the
following properties:

 (i) The smoothness of potentials $v$ is determined only by the rate of decay
of related spectral gaps $\gamma_n = |\lambda_n^+ - \lambda_n^-|, $
where $ \lambda_n^\pm$ are the eigenvalues of $L=L(v)$ considered on
$[0,\pi]$ with periodic (for even $n$) or antiperiodic (for odd $n$)
boundary conditions.

 (ii) There is a Riesz basis
 in $L^2 ([0,\pi], \mathbb{C}^2)$
which consists of periodic (or antiperiodic) eigenfunctions and
 associated functions (at most finitely many).

In particular, the class $X $ contains the families of symmetric
potentials $X_{sym} $ (defined by $\overline{Q} =P$) and
skew-symmetric potentials $X_{skew-sym} $ (defined by $\overline{Q}
=-P$), or more generally the families $X_t, \; t \in
\mathbb{R}\setminus \{0\}, $ defined by $\overline{Q} =t P.$
Finite-zone potentials belonging to $X_t$ are dense in $X_t.$

Another interesting example of potentials  is given by 
$$v=\begin{pmatrix} 0  & P  \\  Q  & 0   \end{pmatrix} \quad
\text{with} \quad
 P(x)  = a e^{ 2ix} + b e^{- 2ix}, \quad Q(x) = A e^{ 2ix} + B
e^{ -2ix}.
$$
If $a, b, A, B \in \mathbb{C}\setminus \{0\},$ 
 then the system of root functions of $L_{Per^\pm}(v)$
 consists eventually of eigenfunctions.
 Moreover, for $bc=Per^-$ this system is a Riesz basis
in $L^2 ([0,\pi], \mathbb{C}^2)$ if  $|aA|=|bB|$ (then $v \in X$), and it is not a
basis if $|aA| \neq |bB|.$
For $bc=Per^+$ the system of root functions is a Riesz basis 
(and $v \in X$) always.

\end{abstract}

\address{Sabanci University, Orhanli,
34956 Tuzla, Istanbul, Turkey}

 \email{djakov@sabanciuniv.edu}

\address{Department of Mathematics,
The Ohio State University,
 231 West 18th Ave,
Columbus, OH 43210, USA} \email{mityagin.1@osu.edu}

\thanks{B. Mityagin acknowledges  the hospitality of Sabanci
University, Istanbul, in May--June, 2009}

\maketitle

\section{Introduction}
We consider one-dimensional Dirac operators of the form
\begin{equation}
\label{001} Ly= i \begin{pmatrix}  1  &  0 \\ 0  & -1 \end{pmatrix}
\frac{dy}{dx} + v(x) \, y, \quad v= \begin{pmatrix}  0  &  P \\ Q  &
0
\end{pmatrix}, \;\; y=\begin{pmatrix}  y_1 \\ y_2 \end{pmatrix}
\end{equation}
with periodic matrix potentials $v$ with $P,Q \in L^2 ([0,\pi],
\mathbb{C}^2),$ subject to periodic ($Per^+$) or antiperiodic
($Per^-$) boundary conditions ($bc$):
\begin{equation}
\label{002} Per^+: \;\;  y(\pi) = y(0); \qquad Per^-: \;\;  y(\pi) =
-y(0).
\end{equation}
Our goal is to single out the class of potentials $v$  which are {\em
special} in the sense that the periodic and antiperiodic boundary
value problems ($b.v.p.$) have at most finitely many linearly
independent associated functions and there is a Riesz basis in $L^2
([0,\pi], \mathbb{C}^2)$ which consists of root functions. It turns
out this is exactly the class of potentials which smoothness could be
determined only by the rate of decay of related spectral gaps
$\gamma_n = |\lambda_n^+ - \lambda_n^-|, $ where $ \lambda_n^\pm$ are
the eigenvalues of $L=L(v)$ considered on $[0,\pi]$ with periodic
(for even $n$) or antiperiodic (for odd $n$) boundary conditions.

Similar questions arise about the one-dimensional Schr\"odinger
operator (e.g., see \cite{DM15,DM25})
\begin{equation}
\label{003} Ly= -y^{\prime \prime} + v(x) y
\end{equation}
with periodic potentials $v \in L^2 ([0,\pi], \mathbb{C}),$ subject
to periodic ($Per^+$) or antiperiodic ($Per^-$) boundary conditions
\begin{equation}
\label{004} Per^\pm: \;\;  y (\pi) = \pm y(0); \qquad y^\prime (\pi)
= \pm y(0).
\end{equation}
Moreover, the methods we use to solve these questions were first
developed for Schr\"odinger operators.

The spectra of self-adjoint Schr\"odinger and Dirac operators with
periodic potentials on the real line $\mathbb{R} $ are continuous and
have gap--band structure: the segments of continuous spectrum
alternate with {\em spectral gaps} or {\em instability zones}. The
theory of Floquet and Lyapunov (e.g., see \cite{E,MW69}) explains that the
end points of spectral gaps are eigenvalues of the same differential
operators but considered on a finite interval of length one period
with periodic or antiperiodic boundary conditions.

The decay rate of spectral gaps depends on  the smoothness of the
potential, and vice versa.  This phenomenon was first
 studied for the Schr\"odinder operator (\ref{003})  with real periodic (say
$\pi$-periodic) potentials $v\in L^2 ([0,\pi]). $ Considered on
$\mathbb{R}$ it generates a self-adjoint operator in $L^2
(\mathbb{R});$ its spectrum is continuous and consists of a sequence
of intervals $ [\lambda_0^+, \lambda_1^-], \,[\lambda_1^+,
\lambda_2^-],\, [\lambda_2^+, \lambda_3^-], \ldots, $ where
$\lambda_0^+ < \lambda_2^- \leq \lambda_2^+ < \lambda_4^- \leq
\lambda_4^- < \cdots $ are all eigenvalues of the periodic $(b.v.p.)$
and $\lambda_1^- \leq \lambda_1^+ < \lambda_3^- \leq \lambda_3^- <
\cdots $ are all eigenvalues of the antiperiodic $b.v.p.$ generated
by $L$  on $[0,\pi].$

 H. Hochstadt \cite{Ho1,Ho2} (see also \cite{LP})
discovered a direct connection between the smoothness of $v$ and the
rate of decay of the lengths of spectral gaps $\gamma_n = \lambda^+_n
-\lambda^-_n :$ {\em If

$(A) \; v \in C^\infty, $ i.e., $v$ is infinitely differentiable,
then

$(B) \; \gamma_n $ decreases more rapidly than any power of $1/n.$

If a continuous function $v$ is a finite--zone potential, i.e.,
$\gamma_n =0$ for large enough $n,$ then} $v \in C^\infty. $\\ In the
mid-70's (see \cite{MO75}, \cite{MT}) the latter statement was
extended, namely, it was shown, for real $L^2 ([0,\pi])$--potentials
$v,$ that  $ \; (B) \Rightarrow (A).$  E. Trubowitz \cite{Tr} has
used the Gelfand--Levitan \cite{GL} trace formula and Dubrovin
equations \cite{Dub75,DMN} to explain, that a real $L^2
([0,\pi])$--potential $v(x)= \sum_{k \in \mathbb{Z}} V(2k) \exp
(2ikx) $ is analytic, i.e.,
$$\exists A>0: \quad |V(2k)| \leq M e^{-A|k|}, $$ if and only if the
spectral gaps decay exponentially, i.e., $$\exists a>0: \quad
\gamma_n \leq C e^{-a|n|}. $$

If the potential $v$ is complex-valued then the Schr\"odinger
operator $L(v) $  is not self-adjoint and one cannot talk about
spectral gaps. But for large enough $n \in \mathbb{N}$ 
 there are two periodic (if $n$ is even) or antiperiodic (if $n$ is
odd) eigenvalues $\lambda^\pm_n $ 
close to $n^2,$
so one may consider ``{\em gaps}``
\begin{equation}
\label{0.0} \gamma_n =|\lambda^+_n - \lambda^-_n|
\end{equation}
and ask whether the rate of decay of $\gamma_n $ still determines the
smoothness of the potential $v.$ The answer to this question is
negative as the example of M. Gasymov \cite{Gas} shows: if
\begin{equation}
\label{0.1}
 v(x) = \sum_{k=0}^\infty v_k e^{2ikx}, \quad v \in L^2([0,\pi])
\end{equation}
then {\em all} eigenvalues of periodic and antiperiodic b.v.p. are
{\em of algebraic multiplicity 2}, so $\gamma_n = 0. $

In \cite{Tk92} V. Tkachenko suggested to consider also the Dirichlet
b.v.p. $y(\pi)=y(0)=0. $  For large enough $n$ 
there is exactly one Dirichlet eigenvalue $\mu_n $ 
close to $n^2, $
so the {\em
deviation}
\begin{equation}
\label{0.2}  \delta_n = |\mu_n - \frac{1}{2}(\lambda^+_n +
\lambda^-_n)|
\end{equation}
is well defined. Using an adequate parametrization of potentials in
spectral terms similar to Marchenko--Ostrovskii's ones
\cite{Mar,MO75} for self-adjoint operators, V. Tkachenko
\cite{Tk92,Tk94}  (see also \cite{ST}) characterized
$C^\infty$-smoothness and analyticity in terms of $\delta_n$ and
differences between critical values of Lyapunov functions and
$(-1)^n$.

T. Kappeler and B. Mityagin \cite{KM1,KM2} suggested a new approach
to the study of spectral gaps and deviations based on Fourier
analysis. Using the Lyapunov-Schmidt reduction method they showed
that for large enough $n$ the numbers $z^\pm_n = \lambda^\pm_n -n^2 $
are the only roots in the unit disc of a quasi-quadratic equation
coined by them as {\em basic equation}
\begin{equation}
\label{0.5}  (z-\alpha_n (z))^2 = \beta^+_n (z) \beta^-_n (z),  \quad
|z|<1,
\end{equation}
where $\alpha_n (z)=\alpha_n (z;v)  $ and $\beta^\pm_n
(z)=\beta^\pm_n (z;v) $ depend analytically on $z,\; |z|<1, $  and
$v$ but the dependance on $v$ is suppressed in the notations. For
large enough $n$ the gaps $\gamma_n $ and deviations $\delta_n $
could be estimated from above in terms of $\beta^+_n (z)$ and $
\beta^-_n (z):$
\begin{equation}
\label{0.6} \exists C>1: \; \gamma_n \leq 2(|\beta^+_n (z)| +
|\beta^-_n (z)|), \quad \delta_n \leq C (|\beta^+_n (z)| + |\beta^-_n
(z)|), \quad   |z|<1.
\end{equation}

Using (\ref{0.6}),  T. Kappeler and B. Mityagin estimated $\ell^2
$-weighted norms $\gamma_n $ and $\delta_n $ by the corresponding
weighted Sobolev norms of $v.$ Let us recall that the smoothness of a
potential $v(x) = \sum_k v_k e^{2 ikx},\; $ can be characterized by
its Fourier coefficients in terms of appropriate weighted norms and
spaces. Namely, if
$$
\omega  = (\omega (k))_{k\in \mathbb{Z}},  \; \omega (-k)= \omega
(k)>0, \quad \omega (0) = 1,
$$
is a weight sequence (or weight), then the corresponding weighted
Sobolev space is
$$ H(\omega) = \left \{  v: \;\;
\|v\|^2_\omega = \sum_{k\in \mathbb{Z}}  |v_k|^2 (\omega (k))^2 <
\infty \right \},
$$ and the corresponding weighted $\ell^2$ space is
$$\ell^2 (\omega, \mathbb{N} ) = \left \{x=(x_n): \;\;  \|x\|^2_\omega
=\sum_{n=1}^\infty  |x_n|^2 (\omega (n))^2 < \infty  \right \}.
$$
Examples of weights:\\
(a)  Sobolev weights: $\;\; \omega_a (0) = 1, \;\; \omega_a (k) =
|k|^a \;\;
\text{for} \;\; k\neq 0; $ \\
(b) Gevrey weights: $\displaystyle\;\; \omega_{b, \gamma} (k) =
e^{b|k|^\gamma}, \quad b>0, \; \gamma \in (0,1);$\\
(c) Abel (exponential) weights: $ \displaystyle\;\; \omega_A  (k) =
e^{A|k|}, \quad A > 0. $

A weight $\Omega $ is called {\em submultiplicative} if
\begin{equation}
\label{0.7} \Omega (k+m) \leq  \Omega (k) \Omega (m), \quad k,m \in
\mathbb{Z}.
\end{equation}

In \cite{KM2}, it was proved that if $\Omega $ is a submultiplicative
weight, then
\begin{equation}
\label{0.8} v \in H(\Omega)  \Rightarrow (|\beta^+_n (z)| +
|\beta^-_n (z)|)\in \ell^2 (\Omega),
\end{equation}
which implies (in view of (\ref{0.6}))
\begin{equation}
\label{0.9} v \in H(\Omega)  \Rightarrow (\gamma_n), (\delta_n) \in
\ell^2 (\Omega ).
\end{equation}

In \cite{GKM} it was suggested to study the spectra of Dirac
operators of the form (\ref{001}) with periodic potentials in a
similar way. If $|n|$ is sufficiently large, then close to $n$ there
is one Dirichlet eigenvalue $\mu_n$ and two periodic (for even $n$)
or antiperiodic (for odd $n$) eigenvalues $\lambda^+_n, \lambda^-_n.$
So, with spectral gaps $\gamma_n $ and deviations $\delta_n $ defined
by
\begin{equation}
\label{0.10} \gamma_n = |\lambda^+_n - \lambda_n^- |, \quad \delta_n
= |\mu_n - \frac{1}{2}(\lambda^+_n - \lambda^-_n)|, \quad n \in
\mathbb{Z}.
\end{equation}
one may study the relationship between potential smoothness and the
rate of decay of $\gamma_n $ and $\delta_n. $  As in the case of
Schr\"odinger operators, there is a basic equation
$$
(z-\alpha_n (z))^2 = \beta^+_n (z) \beta^-_n (z)
$$
which characterizes when $\lambda = z+n $ with $|z|<1/2 $ is a
periodic or antiperiodic eigenvalue, and for large enough $|n|$ the
gaps $\gamma_n $ and deviations $\delta_n $ could be estimated from
above in terms of $\beta^+_n (z)$ and $ \beta^-_n (z)$  by
(\ref{0.6}) -- see below Section~2 for details.

For Dirac potentials $v= \begin{pmatrix} 0 & P\\Q  &0
\end{pmatrix},$  we say $v \in H (\Omega) $ if $P,Q \in H(\Omega). $
Then (\ref{0.9}) holds for Dirac operators: for weights of the form
$\Omega (m) =|m|^a \omega (m) $  with $a\in (0,1/4) $ and
submultiplicative $\omega $ it is proved in \cite{GK01}, and in full
generality (for arbitrary submultiplicative weights $\Omega $) in
\cite{DM6,DM15}.

In \cite{DM3, DM5}, respectively, the authors studied self-adjoint
Schr\"odinger and Dirac operators (i.e., $v$ is real-valued in the
Schr\"odinger case and symmetric, $\overline{Q}=P,$ in the Dirac
case) and estimated the smoothness of potentials $v$ by the rate of
decay of $\gamma_n. $ For a wide classes of weights $\Omega$ it was
shown that
\begin{equation}
\label{0.12} (\gamma_n) \in \ell^2 (\Omega) \Rightarrow  v \in
H(\Omega)
\end{equation}
by proving
\begin{equation}
\label{0.14}    (|\beta^+_n (z)| + |\beta^-_n (z)|) \leq C\gamma_n,
\quad |n| \geq n_0, \; C=2,
\end{equation}
and
\begin{equation}
\label{0.15}    (|\beta^+_n (z)| + |\beta^-_n (z)|)\in \ell^2
(\Omega) \Rightarrow v \in H(\Omega).
\end{equation}
In the non-self-adjoint case -- see \cite{DM5} for Schr\"odinger
operators and \cite{DM7,DM15} for Dirac operators -- we proved that
\begin{equation}
\label{0.16}    (|\beta^+_n (z)| + |\beta^-_n (z)|) \leq C(\gamma_n+
\delta_n), \quad |n| \geq n_0,
\end{equation}
where $C$ is an absolute constant. Of course, (\ref{0.15}) and
(\ref{0.16}) imply that
\begin{equation}
\label{0.18} (\gamma_n), (\delta_n) \in \ell^2 (\Omega) \Rightarrow v
\in H(\Omega).
\end{equation}
 In the
self-adjoint case deviations $\delta_n $ are not important because
the Dirichlet eigenvalue $\mu_n $ is ''trapped'' between
$\lambda^-_n$ and $\lambda^+_n,$  so $\delta_n \leq \gamma_n.$

Our aim in this paper is to study the class $X$ of Dirac potentials
$v$ for which deviations are not essential in the sense that
(\ref{0.14}) holds with some constant $C=C(v). $  A general
criterion is given in Section~3 -- see (\ref{g1}) and
Proposition~\ref{prop1}. It gives non-linear conditions for
individual potentials. Sometimes the family of such potentials is a
real linear space. We observe that an important example of such
linear spaces is the one-parametric family
\begin{equation}
\label{0.21} X_t =\left \{ v = \begin{pmatrix}
 0  &  P  \\ Q  &  0
\end{pmatrix} :  \quad Q = t\overline{P}, \;\; P \in L^2 ([0,\pi])
\right \}, \;\; t \in \mathbb{R}, \;t\neq 0.
\end{equation}
If $t =+1 $ that is the space of symmetric potentials; if $t=-1$ then
we get the space of skew-symmetric potentials.

For any real $t\neq 0 $  we have the following analog of Theorem 58
in \cite{DM15} (more general result is given in Theorem~\ref{thm1} below).
\begin{Theorem}
\label{thm000} Let $$
 L = L^0  + v(x),
\quad L^0 = i
\begin{pmatrix}
 1  &  0 \\
0  &  -1
\end{pmatrix} \frac{d}{dx},
 \quad v(x)=
\begin{pmatrix}
 0  &  P(x) \\
Q(x)  &  0
\end{pmatrix}
$$ be an $X_t$--periodic Dirac operator (i.e., $P$ and $Q$
are periodic  $L^2 ([0,\pi])$--functions such that $Q(x) =
t\overline{P(x)}), $ and let $ \gamma = (\gamma_n)_{n \in \mathbb{Z}}
$ be its gap sequence. If $\Omega = (\Omega (n))_{n\in \mathbb{Z}} $
is a sub--multiplicative weight such that
\begin{equation}
\label{0.25}  \frac{\log \Omega (n)}{n} \searrow 0 \quad \text{as}
\quad n \to \infty,
\end{equation}
then
\begin{equation}
\label{0.26} \gamma \in \ell^2 (\mathbb{Z}, \Omega) \Rightarrow v \in
H (\Omega ).
\end{equation}
If $\Omega $ is a sub--multiplicative weight of exponential type,
 i.e.,
\begin{equation}
\label{0.27} \lim_{n\to \infty} \frac{\log \Omega (n)}{n} >0,
\end{equation}
then there exists $\varepsilon >0 $ such that
\begin{equation}
\label{0.28} \gamma \in \ell^2 (\mathbb{Z}, \Omega) \Rightarrow v \in
H (e^{\varepsilon |n|} ).
\end{equation}
\end{Theorem}

For skew-symmetric potentials  (i.e., when $t= -1$) Theorem~\ref{thm000}
is proved in   \cite{KST} (see Theorem~1.2
and Theorem~1.3 there).  See more comments about   
results and proofs in
\cite{KST} in Section~6 below.
 \vspace{3mm}

In Section 4 we explain that if $v \in X $  then the system of root
functions of the operator $L_{Per^\pm} (v) $ has at most finitely
many linearly independent associated functions and there exists a
Riesz basis in $L^2 ([0,\pi], \mathbb{C}), $ which consists of root
functions. Theorem~\ref{thm00}, which is analogous to Theorem 1 in
\cite{DM25}, gives a necessary and sufficient conditions for
existence of such Riesz bases for a wide class of potentials in $X.$

A real-valued $v$ is called {\em finite-zone potential } if there are
only finitely many $k$ such that $\lambda^-_k < \lambda^+_k.$ S. P.
Novikov \cite{Nov74} raised the question on density of finite-zone
potentials. In 1977 V. A. Marchenko published an article \cite{Ma77}
without proofs, where he gave an explicit construction of a sequence
of finite-zone potentials $v_n $ which converges to a given potential
$v.$  In \cite{MO80} new, simplified proofs were given. To some
extent they have been inspired by the works of T. V. Misyura
\cite{Mi78,Mi79,Mi80,Mi81} on 1D Dirac operators with periodic matrix
potentials.

She considered (in equivalent form) the Dirac operators
\begin{equation}
\label{6} L = iJ\frac{d}{dx} + v, \quad J= \begin{pmatrix} 1&0\\0&-1
\end{pmatrix}, \quad v= \begin{pmatrix} 0&P\\Q&0
\end{pmatrix} \quad P,Q \in
L^2_{loc} (\mathbb{R}), \;  v(x+\pi) = v(x),
\end{equation}
with a symmetric matrix potential $v,$  i.e.,
\begin{equation}
\label{4.1} Q(x) = \overline{P(x)}.
\end{equation}
As in the case of Schr\"odinger operator, $L$ generates a
self-adjoint operator in the space $L^2 (\mathbb{R};\mathbb{C}^2)$ of
$\mathbb{C}^2$-vector functions; its spectrum is continuous and
consists of a sequence of intervals $ [\lambda_{k-1}^+,
\lambda_k^-],\;  k \in \mathbb{Z}, $  where
$$
\cdots < \lambda_{k-1}^+ <  \lambda_k^-  \leq \lambda_k^+ <
\lambda_{k+1}^- < \cdots
$$
are eigenvalues of the periodic b.v.p. $Y(\pi) = Y(0), \; Y(x)=
\begin{pmatrix} y_1 (x) \\ y_2 (x)   \end{pmatrix} $
if $k$ is even, and of the anti-periodic b.v.p.  $Y(\pi) = -Y(0) $ if
$k$ is odd. As in \cite{MO75} the comb domains 
$$  G=\left \{z: \; Im \, z
>0 \right \} \setminus \bigcup_{k \in \mathbb{Z}} [0, h_k] $$ and their
conformal mappings onto the upper half-plane are the essential tool
in \cite{Mi80,Mi81};  there is an one-to-one
correspondence between potentials $v = \begin{pmatrix} 0 & P\\
\overline{P} & 0
\end{pmatrix}$
of the Dirac operators and sequences of real numbers $h = (h_k)_{k\in
\mathbb{Z}},$  $h_k \geq 0, \; \sum h_k^2 < \infty,$ and points $\{
k\pi+i \tilde{h}_k  \},$   $|\tilde{h}_k| \leq h_k.$ Finite-zone
potentials were shown to correspond to sequences with $h_k =0$ for
$|k| \geq N, \; 0 \leq N <\infty. $

If the potential with (\ref{6}) and ({\ref{4.1}) corresponds to the
sequences $(h_k)$ and  $(k\pi+i \tilde{h}_k)$ then the truncated
sequences $(h_k^N)$ and $(k\pi+i \tilde{h}^N_k),$ where
$$
h_k^N =\begin{cases} h_k &   0 \leq |k| \leq N  \\  0 & |k|>N,
\end{cases}
\qquad \text{and} \qquad
\tilde{h}^N_k =\begin{cases} \tilde{h}_k &   0 \leq |k| \leq N  \\
0 & |k|>N,
\end{cases}
$$
correspond to the $(2N+2)$-zone potential  $v_N (x)=\begin{pmatrix} 0
& P_N (x)  \\ \overline{P_N (x)}& 0
\end{pmatrix}$
and
$$
\|P - P_N\|_{L^2([0,\pi])} \leq \|h-h^N \| \cdot (1+2\|h-h^N \|)
C(\|h\|)
$$
where $ C(x) = 16 \sqrt{\pi} (1+\pi^2/2)^5 e^{7x}, \; x>0.$

If the potential $v$  in (\ref{6}) is not symmetric then the methods
of \cite{MO80} and \cite{Mi80,Mi81} can not be applied directly.

V. A. Tkachenko \cite{Tk00} considered skew-symmetric potentials
$v(x)=i \begin{pmatrix} 0 & P \\ \overline{P} &  0 \end{pmatrix}.$ In
this class he proved that finite-zone skew symmetric potentials are
dense.

(Of course, in the non-symmetric case the notion of finite-zone
potential should be properly adjusted. A potential $v \in (\ref{6}) $
is {\em finite-zone} if for all but finitely many $n \in \mathbb{Z} $
$$   \lambda^+_n =   \lambda^-_n = \mu_n, $$
where $\mu_n $ is a Dirichlet eigenvalue such that $|\mu_n - n|
<1/4).$

In 2000 B. Mityagin \cite{Mit00} suggested (at least in the
Schr\"odinger-Hill case) an approach to construction of potentials
with prescribed tails of their spectral gap sequences. In particular,
if the tails are zero sequences one gets finite-zone potentials.
(With more careful analysis of the eigenvalues of the operator $L$
this approach leads to construction of potentials -- both for
Schr\"odinger-Hill  and Dirac operators -- whose eigenfunction
expansions do not converge in  $L^2.$ For details see \cite[Theorem
71 and Section 5.2]{DM15}.)

It turns out that the same method works for Dirac operators as well.
Following the scheme of \cite{Mit00} B. Grebert and T.
Kappeler\footnote{They wrote (see \cite{GK03}): "To prove Theorem 1.1
... we follow the approach used in \cite{Mit00}: as a set-up we take
the Fourier block decomposition introduced first for the Hill
operator in \cite{KM1,KM2} and used out subsequently for the
Zaharov--Shabat operators in \cite{GKM,GK01}. Unlike in \cite{Mit00}
where a contraction mapping argument was used to obtain the density
results for the Hill operator, we get a short proof of Theorem 1.1 by
applying the inverse function theorem in a straightforward way. As in
\cite{Mit00}, the main feature of the present proof is that it does
not involve any results from the inverse spectral theory."}
\cite{GK03} proved the density of finite-zone potentials in the
spaces $H(\Omega) $ (see Definition 2 in Section 2) under the
restriction $H(\Omega)\subset H^a, \; \exists a>0, $ where $H^a $ is
a Sobolev space; in general, the density of finite-zone potentials in
the spaces $H(\Omega) $ was proved by P. Djakov and B. Mityagin (see
an announcement in \cite{Mit03}, and a complete proof in
\cite[Theorem 70]{DM15}).

We explain in Section 5 that the proof of Theorem 70 in \cite{DM15}
as it is written there covers not only the general and symmetric
cases but a broad range of linear and nonlinear families of
potentials; certainly, among them is the space of skew-symmetric
potentials $v = i\begin{pmatrix}
 0  &  P  \\ \overline{P}  &  0
\end{pmatrix}. $

The finite-zone potential density results announced in \cite{Mit03}
and proved in \cite{DM15} for general potentials and symmetric
potentials $(v^*=v)$ could be extended immediately for
skew-symmetric potentials and $X_t$-potentials as well if one
notices that all the (non-linear) operators  $\Phi_N, \, A_N $ (see
below (\ref{4.31}) and (\ref{4.32})) act in the space of general
potentials
$$  \left \{v=
\begin{pmatrix} 0 & P
\\ Q & 0 \end{pmatrix},\quad P,Q \in L^2 ([0,\pi]) \right \} $$
in such a way that both
$$ X_{sym} = \{v \in X: \; Q(x)= \overline{P(x)}  \} $$ and
$$ X_{skew-sym} = \{v\in X:  \;\;  P(x)= i R(x), \;\; Q(x) = i
\overline{R(x)} \} $$ and any $X_t $  are {\em invariant} for these
operators.

\section{Preliminaries}

The Dirac operator (\ref{001}), considered on the interval $[0,\pi]$
with periodic $Per^+,$  antiperiodic $Per^-$ and Dirichlet $Dir$
boundary conditions $(bc)$ $$ Per^{\pm}:  \; y(\pi) = \pm y(0), \quad
Dir: \; y_1 (0)=y_2 (0), \; y_1 (\pi)=y_2 (\pi), $$ with $y(x) =
\begin{pmatrix} y_1 (x) \\ y_2 (x)  \end{pmatrix},$ gives a rise
of three operators $L_{bc} (v), \; bc = Per^\pm, Dir. $ Their spectra
are discrete; moreover, the following holds.

\begin{Lemma} \label{loc}
({\em Localization Lemma.})  The spectra of $L_{bc} (v), \;
bc=Per^\pm, \; Dir $ are discrete. There is an $N=N(v)$  such that
the union $\cup_{|n|>N} D_n $ of the discs $D_n =\{z: \, |z-n|< 1/4
\}$ contains all but finitely many of the eigenvalues of $L_{bc}, \;
bc=Per^\pm, \; Dir $ while the remaining finitely many eigenvalues
are situated in the rectangle $R_N = \{z: \; |Re \,z|, \, |Im \, z|
\leq N+1/2 \}.$

Moreover, for $|n|>N$ the disc $D_n$ contains one Dirichlet
eigenvalue $\mu_n $ and two (counted with algebraic multiplicity)
periodic (if $n$ is even) or antiperiodic (if $n$ is odd) eigenvalues
$\lambda_n^-, \lambda_n^+$  (where $Re \,\lambda_n^- < Re
\,\lambda_n^+$  or $Re \,\lambda_n^- = Re \,\lambda_n^+$ and $Im
\,\lambda_n^- \leq Im \,\lambda_n^+).$
\end{Lemma}

See details and more general results about localization of these
spectra in \cite{Mit03,Mit04} and \cite[Section 1.6.]{DM15}.

Now, in view of Lemma \ref{loc}, for $|n| > N(v)$  the spectral gaps
\begin{equation}
\label{13.3} \gamma_n = |\lambda_n^+ - \lambda_n^-|
\end{equation}
and deviations
\begin{equation}
\label{13.4} \delta_n = |\mu_n - \frac{1}{2}(\lambda_n^+ +
\lambda_n^-)|
\end{equation}
are well-defined.

 Moreover, the localization Lemma \ref{loc} allows
us to apply the Lyapunov--Schmidt projection method and reduce the
eigenvalue equation $Ly = \lambda y $ for  $\lambda \in D_n $ to an
eigenvalue equation in the two-dimensional space $E_n^0 = \{ L^0 Y =
n Y\}$   (see \cite[Section 2.4]{DM15}).

This leads to the following (see in \cite{DM15} the formulas
(2.59)--(2.80) and Lemma~30).
\begin{Lemma}
\label{lem1} Let $$ P(x) = \sum_{k\in 2\mathbb{Z}} p(k) e^{ikx},
\quad  Q(x)= \sum_{k\in 2\mathbb{Z}} q(k) e^{ikx}, $$ and let
\begin{equation}
\label{23.29} S^{11} = \sum_{\nu =0}^\infty S^{11}_{2\nu +1} , \quad
S^{22} = \sum_{\nu =0}^\infty S^{22}_{2\nu +1} , \qquad S^{12} =
\sum_{\nu =1}^\infty S^{12}_{2\nu} , \quad S^{21} = \sum_{\nu
=1}^\infty S^{21}_{2\nu} ,
\end{equation}
where
\begin{equation}
\label{23.30} S^{11}_{2\nu +1} =
 \sum_{j_0 , j_1 , \ldots, j_{2\nu} \neq n }
\frac{p(-n-j_0 ) q(j_0 + j_1 ) p(-j_1 -j_2 ) q(j_2 + j_3 ) \cdots
 q(j_{2\nu} + n )} {(n-j_0  +z) (n-j_1 +z)
 \cdots   (n-j_{2\nu}  +z)},
\end{equation}
\begin{equation}
\label{23.31} S^{22}_{2\nu +1} =  \sum_{i_0 , i_1 , \ldots, i_{2\nu}
\neq n } \frac{q(n + i_0 ) p(-i_0 -i_1 ) q(i_1 + i_2 ) p( -i_2 -i_3 )
\cdots  p(-i_{2\nu} - n )} {(n-i_0  +z) (n-i_1  +z) \ldots
(n-i_{2\nu} +z)};
\end{equation}
\begin{equation}
\label{23.34} S^{12}_0 = \left  \langle   V e^2_n, e^1_n \right
\rangle = p(-2n), \quad S^{21}_0 = \left  \langle   V e^1_n, e^2_n
\right \rangle = q(2n),
\end{equation}
and, for $\nu =1,2 \ldots, $
\begin{equation}
\label{23.36} S^{12}_{2\nu} =\sum_{j_1 , \ldots, j_{2\nu} \neq n }
\frac{p(-n-j_1 ) q(j_1 + j_2 ) p(-j_2 -j_3 ) q(j_3 + j_4 ) \cdots
 p(-j_{2\nu} - n )} {(n-j_1  +z) (n-j_2
+z) \cdots   (n-j_{2\nu}  +z)},
\end{equation}
\begin{equation}
\label{23.37} S^{21}_{2\nu} =
 \sum_{j_1 , \ldots, j_{2\nu} \neq n } \frac{q(n+j_1 ) p(-j_1 -
j_2 ) q(j_2 +j_3 ) p(-j_3 - j_4 ) \cdots  q(j_{2\nu} + n )} { (n-j_1
+z) (n-j_2  +z) \cdots  (n-j_{2\nu} +z)}.
\end{equation}

(a) For large enough $|n|$ the series in (\ref{23.29})-(\ref{23.37})
converge absolutely and uniformly if $|z|\leq 1, $ so $ S^{ij} (n, z,
p, q) $ are analytic functions of $z$ for $|z|< 1. $

 (b) The number $\lambda = n + z, \;|z| < 1/4, $ is a periodic (for
even $n$) or antiperiodic (for odd $n$) eigenvalue of $L$ if and only
if $z$  is an eigenvalue of the matrix $ \left [
\begin{array}{cc} S^{11} & S^{12}  \\ S^{21} &  S^{22} \end{array}
\right ]. $

(c) The number $\lambda = n + z^*, \;|z| < 1/4, $ is a periodic (for
even $n$) or antiperiodic (for odd $n$) eigenvalue of $L$ of
geometric multiplicity 2  if and only if $z^*$  is an eigenvalue of
the matrix $ \left [
\begin{array}{cc} S^{11} & S^{12}  \\ S^{21} &  S^{22} \end{array}
\right ] $ of geometric multiplicity 2.
\end{Lemma}

 Moreover,
(\ref{23.29})--(\ref{23.37}) imply immediately the following.

\begin{Lemma}
\label{lem2} (a) For any potential functions $P, Q  $
 \begin{equation}
\label{1.16}  S^{11}(n,z;p,q)=S^{22}(n,z;p,q), \quad S^{21}
(n,z;p,q)= \overline{S^{12} (n,\overline{z}; \overline{q},
\overline{p} )},
\end{equation}
\begin{equation}
\label{1.17} S_{2\nu}^{21} (n,z;tp,sq) = t^\nu s^{\nu+1}S_{2\nu}^{21}
(n,z;p,q), \quad S_{2\nu}^{12} (n,z;tp,sq) = t^{\nu+1} s^\nu
S_{2\nu}^{12} (n,z;p,q)
\end{equation}
\begin{equation}
\label{1.19} S_{2\nu+1}^{jj} (n,z;tp,sq) = t^{\nu+1} s^{\nu+1}
S_{2\nu+1}^{jj} (n,z;p,q), \quad j=1,2.
\end{equation}

(b) If $Q(x) = c \overline{P(x)}, \;  $  $c$ real, then
(\ref{1.16})-(\ref{1.19}) imply
\begin{equation}
\label{1.20} \overline{S^{21} (n,\overline{z}; p,q)} = c S^{12}
(n,z;p,q), \quad \overline{S^{jj} (n,z;p,q)} = S^{jj}
(n,\overline{z}, p,q).
\end{equation}

(c) In the case of skew-symmetric potentials $ c=-1,$  so
\begin{equation}
\label{1.21} \overline{S^{21}(n,\overline{z})} =- S^{12} (n,z).
\end{equation}
\end{Lemma}

We set for convenience
\begin{equation}
\label{1.22} \alpha_n (z;v) := S^{11} (n,z;v) \quad \beta^+_n (z;v)
:= S^{21} (n,z;v), \quad \beta^-_n (z;v) := S^{12} (n,z;v).
\end{equation}
Next we summarize some basic properties of $\alpha_n (z;v) $ and
$\beta^\pm_n (z;v).$

\begin{Proposition}
\label{bprop} (a) The functions $\alpha_n (z;v) $ and $\beta^\pm_n
(z;v) $ depend analytically on $z$ for $|z|\leq 1$  and  the
following estimates hold:
\begin{equation}
\label{1.36} |\alpha_n (v;z)|, \, |\beta^\pm_n (v;z)| \leq C \left (
\mathcal{E}_{|n|} (r) +\frac{1}{\sqrt{|n|}} \right ) \quad \text{for}
\; |n| \geq n_0, \;\; |z| \leq \frac{1}{2}
\end{equation}
and
\begin{equation}
\label{1.37} \left |\frac{\partial\alpha_n}{\partial z} (v;z)
\right|, \, \left |\frac{\partial\beta^\pm_n}{\partial z} (v;z)
\right |\leq C \left ( \mathcal{E}_{|n|} (r) +\frac{1}{\sqrt{|n|}}
\right ) \quad \text{for} \;\; |n| \geq n_0, \;\; |z| \leq
\frac{1}{4},
\end{equation}
where $ r= (r(m)), \; r(m) = \max \{|p (\pm m)|, q(\pm m) \}, \; C =
C(\|r\|), \; n_0 =n_0 (r)$  and
$$
\mathcal{E}_m (r) = \left (\sum_{|n|\geq m} |r(n)|^2  \right )^{1/2}.
$$

(b) For large enough $n,$ the number $\lambda = n+ z, $ $ z\in
D=\{\zeta: |\zeta| \leq 1/4\}, $ is an eigenvalue of $L_{Per^\pm} $
if and only if
 $z\in D $ satisfies the basic equation
\begin{equation}
\label{be} (z-\alpha_n (z;v))^2 = \beta^+_n (z;v) \beta^-_n (z,v),
\end{equation}

(c) For large enough $n,$ the equation (\ref{be}) has exactly two
roots in $D$ counted with multiplicity.
\end{Proposition}

\begin{proof}
Part (a) is proved in \cite[Proposition 35]{DM15}. Lemma~\ref{lem1}
implies Part (b).  By (\ref{1.36}), $\sup_D |\alpha_n (z)| \to 0 $
and $\sup_D |\beta^\pm_n (z)| \to 0 $ as $n\to \infty. $ Therefore,
Part (c) follows from the Rouch\'e theorem.
\end{proof}

In view of Lemma \ref{loc}, for large enough $|n|$  the numbers
\begin{equation}
\label{1.41} z^*_n = \frac{\lambda^+_n + \lambda_n^-}{2} - n
\end{equation}
are well defined. The following estimate from above of $\gamma_n $
follows from Proposition~\ref{bprop} (see \cite[Lemma 40]{DM15}).

\begin{Lemma}
\label{lem11}  For large enough $|n|$
\begin{equation}
\label{1.42} \gamma_n =|\lambda^+_n - \lambda_n^-| \leq  (1+\delta_n)
(|\beta_n^- (z^*_n) |+|\beta^+_n (z^*_n)|)
\end{equation}
with $\delta_n \to 0 $ as $|n| \to \infty. $
\end{Lemma}

\section{Spectral gaps asymptotics and potential smoothness}

 Let $X$ be the class of all Dirac potentials $v =
\begin{pmatrix} 0 & P
 \\  Q&  0   \end{pmatrix} $  such that
\begin{equation}
\label{g1} \exists c, \, N>0 : \quad c^{-1} |\beta^+_n (z^*_n;v)|
\leq |\beta^-_n (z^*_n;v)| \leq  c \,|\beta^+_n (z^*_n;v)|,  \quad
|n|>N.
\end{equation}

\begin{Lemma}
\label{lem21} Suppose $v \in X$ and the set $M$ of all $n\in
\mathbb{Z}$ such that $\beta^-_n (z^*_n;v)\neq 0 $ and $\beta^+_n
(z^*_n;v)\neq 0 $ is infinite. Let $K_n $  be the closed disc with
center $z_n^*$ and radius $\gamma_n, $  i.e. $K_n = \{z:
|z-z_n^*|\leq \gamma_n \}.$ Then for all $n \in M$ with sufficiently
large $|n|$ we have
\begin{equation}
\label{g2} \frac{1}{2}|\beta^\pm_n (z_n^*;v)|  \leq   |\beta^\pm_n
(z;v)| \leq 2 |\beta^\pm_n (z_n^*;v)| \quad \forall \,z \in K_n,
\end{equation}
where $c$ is the constant from (\ref{g1}).
\end{Lemma}

\begin{proof}
In view of (\ref{1.37}) in Proposition \ref{bprop},  if $z \in K_n $
then
$$
\left | \beta^\pm_n (z)- \beta^\pm_n (z_n^*) \right | \leq
\varepsilon_n  \left |z - z_n^* \right | \leq \varepsilon_n \,
\gamma_n,
$$
where $\varepsilon_n=C \left ( \mathcal{E}_{|n|} (r)
+\frac{1}{\sqrt{|n|}} \right ) \to 0 $  as $|n| \to \infty.$ By
Lemma~\ref{lem11}, for large enough $|n|$ we have
$$
\gamma_n \leq 2  \left ( |\beta^-_n (z_n^*)| +|\beta^+_n (z_n^*)|
\right ).
$$
Therefore,
$$
\left | \beta^\pm_n (z)- \beta^\pm_n (z_n^*) \right | \leq
2\varepsilon_n  \left ( |\beta^-_n (z_n^*)| +|\beta^+_n (z_n^*)|
\right ) \leq 2\varepsilon_n (1+c) \left | \beta^\pm_n (z_n^*) \right
|,
$$
which implies
$$
[1-2\varepsilon_n (1+c)]\left | \beta^\pm_n (z_n^*)  \right | \leq
\left | \beta^\pm_n (z)  \right | \leq [1+2\varepsilon_n (1+c)]\left
| \beta^\pm_n (z_n^*)  \right |.
$$
Since $\varepsilon_n \to 0$ as $|n| \to \infty, $ (\ref{g2}) follows.
\end{proof}

\begin{Proposition}
\label{prop1} Suppose $v$ is a Dirac potential such that (\ref{g1})
holds. Then, for $|n|>N_0 (v),$ the following two-sided estimates for
$\gamma_n = |\lambda^+_n - \lambda^-_n|$ hold:
\begin{equation}
\label{g4} \frac{2\sqrt{c}}{1+4c}\,\left (|\beta^-_n
(z^*_n;v)|+|\beta^+_n (z^*_n;v)|\right ) \leq \gamma_n \leq 2 \left
(|\beta^-_n (z^*_n;v)|+|\beta^+_n (z^*_n;v)|\right ).
\end{equation}
\end{Proposition}

\begin{proof} The estimate of $\gamma_n $ from above
follows from Lemma~\ref{lem11}.

In view of  (\ref{g1}),  $\beta^+_n (z^*_n;v)$  and $\beta^-_n (z^*_n;v)$
may vanish only simultaneously.
Suppose that $\beta^+_n (z^*_n;v)\cdot \beta^-_n (z^*_n;v) \neq 0 $
for infinitely many $n$    --  for such $n$ we have $\gamma_n \neq 0 $
due to Lemma~\ref{lem1}(c).
Then, by Lemma~49 in \cite{DM15}, there
exists a sequence $\delta_n \downarrow 0 $ such that, for large
enough $|n|, $
\begin{equation}
\label{g5}  \gamma_n  \geq \left (\frac{2\sqrt{t_n}}{1+t_n}
-\delta_n \right )  \left ( |\beta^-_n (z_n^*)| +|\beta^+_n (z_n^*)|
\right ),
\end{equation}
where $\delta_n \to 0 $ as $|n| \to \infty $ and
$$
 t_n = |\beta_n^+ (z_n^+)| /|\beta_n^- (z_n^+)|, \quad z_n^+ =
 \lambda_n^+ - n.
$$

In view of (\ref{g2}) in Lemma \ref{lem21},  for large enough $|n|$
we have $    1/(4c)  \leq t_n  \leq 4c. $ Therefore, by (\ref{g5}),
$$
 \gamma_n  \geq \left ( \frac{2\sqrt{4c}}{1+4c}-\delta_n \right )
  \left ( |\beta^-_n (z_n^*)| +|\beta^+_n (z_n^*)|
\right ),
$$
which implies (since $\delta_n \to 0$) the left inequality in
(\ref{g4}). This completes the proof.
\end{proof}

\begin{Corollary}
\label{cor3} If $v \in X$  then  the operators $L_{Per^\pm}$ have at
most finitely many eigenvalues of algebraic multiplicity 2 but
geometric multiplicity 1.
\end{Corollary}

\begin{proof}
Indeed, the estimate in (\ref{g4}) imply that for large enough $|n|$
the number $\lambda^*_n = \lambda^+_n = \lambda^-_n$ is a double
eigenvalue if and only if $\beta^+_n (z^*_n)=\beta^-_n (z^*_n)=0. $
But then, in view of (\ref{1.41}), the number $z_n^*  $ is a double
root of the basic equation (\ref{be}), so it is a double eigenvalue
of the matrix $\left [\begin{array}{cc} \alpha_n (z^*_n) & \beta^-_n
(z^*_n)\\ \beta^+_n (z^*_n)& \alpha_n (z^*_n)
\end{array}\right ]$ of geometric multiplicity 2 because the
off-diagonal elements are zeros. By Lemma~\ref{lem1}, the number
$\lambda_n^*=z_n^* +n $ is a double eigenvalue of $L(v)$ of geometric
multiplicity 2 (periodic for even $n$ or antiperiodic for odd $n$),
so the corresponding two-dimensional invariant subspace consists of
eigenvectors only.
\end{proof}

A sequence of positive numbers  $$\Omega (m), \, m \in \mathbb{Z},
\quad \Omega (-m) =\Omega (m), $$ is called submultiplicative weight
sequence (or submultiplicatve weight) if $$ \Omega (n+m) \leq \Omega
(n) \Omega (m), \quad n,m \in \mathbb{Z}. $$ For any
submultiplicative weight we define the Hilbert sequence space $$
\ell^2 (\Omega) = \{(x_k)_{k \in \mathbb{Z}} \;: \;\; \sum_k |x_k|^2
(\Omega (k))^2 < \infty \}$$ and the functional space
\begin{equation}
\label{34.1} H(\Omega) = \{f= \sum f_k e^{i2kx}\; : \;\; \sum_k
|f_k|^2 (\Omega (k))^2 < \infty\}.
\end{equation}
We consider also the weighted Hilbert space of potentials
\begin{equation}
\label{34.2} H_D (\Omega)= \left \{v = \begin{pmatrix}
 0  &  P  \\Q  & 0  \end{pmatrix} \;: \quad  P,Q \in H(\Omega) \right \}.
\end{equation}

By \cite[Theorem 41]{DM15}, if $\Omega $ is a submultiplicative
weight, then
\begin{equation}
\label{34.3} v \in H_D (\Omega)  \Rightarrow  (\gamma_n ) \in \ell^2
(\Omega).
\end{equation}
The converse implication
\begin{equation}
\label{34.4} (\gamma_n ) \in \ell^2 (\Omega) \Rightarrow v \in H_D
(\Omega)
\end{equation}
holds in the self-adjoint case where $\overline{Q(x)} = P(x)$ under
some additional assumptions on $\Omega $ (see Theorem 58 in
\cite{DM15}) but fails in general (see however Theorem 68 in
\cite{DM15}). The following statement extends the validity of
(\ref{34.4}) to the case where $v \in X.$

\begin{Theorem}
\label{thm1} Let $$
 L = L^0  + v(x),
\quad L^0 = i
\begin{pmatrix}
 1  &  0 \\
0  &  -1
\end{pmatrix} \frac{d}{dx},
 \quad v(x)=
\begin{pmatrix}
 0  &  P(x) \\
Q(x)  &  0
\end{pmatrix}
$$ be a Dirac operator with potential $v \in X $
and let $ \gamma = (\gamma_n)_{n \in \mathbb{Z}} $ be its gap
sequence. If $\Omega = (\Omega (n))_{n\in \mathbb{Z}} $ is a
sub--multiplicative weight such that
\begin{equation}
\label{34.05}  \frac{\log \Omega (n)}{n} \searrow 0 \quad \text{as}
\quad n \to \infty,
\end{equation}
then
\begin{equation}
\label{34.06} \gamma \in \ell^2 (\mathbb{Z}, \Omega) \Rightarrow v
\in H (\Omega ).
\end{equation}
If $\Omega $ is a sub--multiplicative weight of exponential type,
 i.e.,
\begin{equation}
\label{34.100} \lim_{n\to \infty} \frac{\log \Omega (n)}{n} >0,
\end{equation}
then there exists $\varepsilon >0 $ such that
\begin{equation}
\label{34.08} \gamma \in \ell^2 (\mathbb{Z}, \Omega) \Rightarrow v
\in H (e^{\varepsilon |n|} ).
\end{equation}
\end{Theorem}

\begin{proof}
In view of Proposition~\ref{prop1}, if $ \gamma =(\gamma_n)_{n\in
\mathbb{Z}}\in \ell^2 (\mathbb{Z}, \Omega)$ then we have $$ \left
(|\beta_n^- (v, z_n^*)|+|\beta_n^- (v, z_n^*)| \right )_{|n|>N} \in
\ell^2 (\Omega).
$$
In other notations,
\begin{equation}
\label{34.10}
\gamma \in \ell^2 (\mathbb{Z}, \Omega) \Rightarrow A_N (v) \in H
(\Omega ),
\end{equation}
where (compare with \cite[(3.52)--(3.54)]{DM15}) the nonlinear
operators $A_N $ are defined by
$$
A_N (v) = v + \Phi_N (v), \quad \Phi_n (v) = \begin{pmatrix} 0 &
\Phi_N^{12}  \\ \Phi^{21} (v) &  0
\end{pmatrix}
$$
with
$$
\Phi_N^{12}= \sum_{|n|>N} (\beta_n^- (z_n^*,v) - p(-n))e^{-2inx}
\quad \text{and} \quad \Phi_N^{21}= \sum_{|n|>N} (\beta_n^+ (z_n^*,v)
- q(n))e^{2inx}.
$$
Now  Theorem \ref{thm1} follows from  
\cite[Lemma 48 and Proposition 57]{DM15}  in the same way as \cite[Theorem 58]{DM15}
(self-adjoint case) -- namely,  by (\ref{34.10}) and Lemma 48 there is a slowly growing wight $\Omega_1 $
such that $A_N (v) \in H(\Omega \cdot \Omega_1),$ 
so Proposition~57  implies  $v \in H(\Omega \cdot \Omega_1) \subset H(\Omega).$  
 
\end{proof}

\section{Riesz bases}
Let $H$ be a Hilbert space. A family of bounded finite--dimensional
projections $\{P_\gamma: H\to H , \, \gamma \in \Gamma\}$ is called
{\em  basis of projections} if
\begin{align}
P_\alpha P_\beta =0  \quad \text{if} \;\; \alpha \neq \beta;\label{r1}\\
x= \sum_{\gamma \in \Gamma} P_\gamma (x) \quad \forall x \in
H,\label{r2}
\end{align}
where the series converge in $H.$

If $(Q_\gamma)$ is a basis of orthogonal projections (i.e.,
$Q_\gamma^* = Q_\gamma $),  the Pythagorian theorem implies  $
\sum_\gamma \|Q_\gamma x\|^2 = \|x\|^2. $

  A family of projections $(P_\gamma, \, \gamma\in \Gamma)$ is
 called {\em Riesz basis of projections} if
\begin{equation}
\label{r3} P_\gamma = A Q_\gamma A^{-1}, \quad \gamma\in \Gamma,
\end{equation}
where $A:H\to H $ is an isomorphism and $(Q_\gamma, \, \gamma\in
\Gamma)$ is a basis of orthogonal projections.

It is well known (see G-K) that a basis of projections $(P_\gamma, \,
\gamma\in \Gamma)$ is a Riesz basis of projections if and only if
there are constants $a, b>0 $ such that
\begin{equation}
\label{r4}       a \|x\|^2  \leq \sum_\gamma \|P_\gamma x\|^2 \leq b
\|x\|^2 \quad x \in H
\end{equation}
(equivalently, if and only if the family $\{P_\gamma, \, \gamma \in
\Gamma\} $ is orthogonal with respect to an equivalent Hilbert norm).

A family of vectors $\{f_\gamma, \, \gamma \in \Gamma\}$ is called a
basis in $H$ if
\begin{equation}
\label{r5} x= \sum_{\gamma \in \Gamma} c_\gamma (x) f_\gamma \quad
\forall x \in H,
\end{equation}
where the series converge in $H$ and the scalars $c_\gamma (x)$ are
uniquely determined.

Obviously, if $(f_\gamma) $ is a basis in $H$ then the system of
one--dimensional projections $P_\gamma (x) = c_\gamma (x) f_\gamma$
is a basis of projections in $H, $ and vice versa, every basis of one
dimensional projections can be obtained in that way from some basis
of vectors.

A system of vectors $\{f_\gamma, \, \gamma \in \Gamma\}$ is called
Riesz basis in $H$ if it has the form
\begin{equation}
\label{r6} f_\gamma = A \,e_\gamma, \quad \gamma \in \Gamma,
\end{equation}
where $A$ is an isomorphism $A: H \to H $ and $e_\gamma, \, \gamma
\in \Gamma$ is an orthonormal basis in $H.$

A basis $\{f_\gamma, \, \gamma \in \Gamma\}$ is a Riesz basis if and
only if there are constants $a, b, c, C >0 $ such that
\begin{equation}
\label{r7a}      c \leq \|f_\gamma \| \leq C  \quad  \forall  \gamma \in \Gamma, \quad  
a \|x\|^2  \leq \sum_\gamma |c_\gamma (x)|^2 \leq b
\|x\|^2,  \quad x \in H
\end{equation}
(equivalently, if and only if the family $\{f_\gamma, \, \gamma \in
\Gamma\} $ is orthogonal with respect to an equivalent Hilbert norm
and $0< \inf \|f_\gamma\|, \;  \sup \|f_\gamma\| < \infty$).

\begin{Lemma}
\label{lemr} Let $(P_\gamma, \, \gamma\in \Gamma)$  be a Riesz basis
of two-dimensional projections in a Hilbert space $H,$  and let
  $f_\gamma, \, g_\gamma \in Ran \,
P_\gamma, $  $\gamma \in \Gamma$   are linearly independent unit
vectors. Then the system $\{f_\gamma, \, g_\gamma, \; \gamma \in
\Gamma \} $ is a Riesz basis if and only if
\begin{equation}
\label{r7}  \kappa :=\sup |\langle  f_\gamma,  g_\gamma \rangle | <
1.
\end{equation}
\end{Lemma}

\begin{proof}
If the system $\{f_\gamma, \, g_\gamma, \; \gamma \in
\Gamma \} $ is a Riesz basis in $H, $ then
$$
x= \sum_\gamma (f^*_\gamma (x) f_\gamma + g^*_\gamma (x) g_\gamma),
\quad x \in H,
$$
where $f^*_\gamma, g^*_\gamma$ are the conjugate functionals. In view of 
(\ref{r7a}),  the one-dimensional projections
$$
P^1_\gamma (x) =f^*_\gamma (x) f_\gamma,  \quad P^2_\gamma
(x)=g^*_\gamma (x) g_\gamma
$$
are uniformly bounded. On the other hand, it is easy to see that
$$
\|P^1_\gamma\|^2 \geq \left ( 1- |\langle  f_\gamma,  g_\gamma
\rangle |^2 \right )^{-1}, \quad \|P^2_\gamma\|^2 \geq \left (1-
|\langle  f_\gamma,  g_\gamma \rangle |^2 \right )^{-1},$$ so
(\ref{r7}) holds.

Conversely, suppose (\ref{r7}) holds. Then we have for every $\gamma
\in \Gamma$
$$
(1-\kappa) \left (  |f^*_\gamma (x)|^2 +|g^*_\gamma (x)|^2 \right )
\leq \|P_\gamma (x)\|^2 \leq (1+\kappa) \left (  |f^*_\gamma (x)|^2
+|g^*_\gamma (x)|^2 \right )
$$
which implies, in view of (\ref{r4}),
$$
\frac{a}{1+\kappa} \|x\|^2  \leq \sum_\gamma \left (  |f^*_\gamma
(x)|^2 +|g^*_\gamma (x)|^2 \right ) \leq \frac{b}{1-\kappa} \|x\|^2.
$$
Therefore, (\ref{r7}) holds, which means that the system $\{f_\gamma,
\, g_\gamma, \; \gamma \in \Gamma \} $ is a Riesz basis in $H.$
\end{proof}

In view of Lemma~\ref{loc}, the Dirac operators (\ref{001})
 with  $L^2$-potentials
$$ v(x) = \begin{pmatrix} 0 & P(x) \\ Q(x) & 0 \end{pmatrix},
 \quad P,Q \in L^2 ([0,\pi]), $$ considered on $[0,\pi]$
with periodic or antiperiodic boundary conditions  have discrete
spectra, and the Riesz projections
\begin{equation}
\label{r11} S_N = \frac{1}{2\pi i} \int_{ \partial R_N }
(z-L_{Per^\pm})^{-1} dz, \quad P_n = \frac{1}{2\pi i} \int_{ |z-n|=
\frac{1}{4} } (z-L_{Per^\pm})^{-1} dz
\end{equation}
are well--defined for  $|n| \geq N$ if $N $ is sufficiently large.

By \cite[Theorem 3]{DM20}),
\begin{equation}
\label{r12} \sum_{|n| > N} \|P_n - P_n^0\|^2 < \infty,
\end{equation}
where $P_n^0, \, n \in \mathbb{Z},$ are the Riesz projections of the
free operator. Moreover, the Bari--Markus criterion implies (see
Theorem 9 in \cite{DM20}) that the spectral Riesz decompositions
\begin{equation}
\label{r13}
 f = S_N f +  \sum_{|n| >N} P_n f, \quad  \forall
 f \in L^2 \left ([0,\pi], \mathbb{C}^2 \right ),
\end{equation}
converge unconditionally.  In other words, $\{S_N, \; P_n, \, |n|>N\}
$ is a Riesz projection basis in the space $L^2 \left ([0,\pi],
\mathbb{C}^2 \right ).$

Each of the projections $P_n, \; |n|>N,$ is two-dimensional, and if
$v \in X $ then for large enough $N$ each two-dimensional block $Ran
\, P_n$ consists of eigenfunctions only. In the next theorem, we show
that if $v \in X, $ then it is possible to build a Riesz basis of
eigenfunctions in $H = \bigoplus_{|n|>N} Ran (P_n) $ by "splitting"
two-dimensional blocks $Ran (P_n). $

\begin{Theorem}
\label{thm0}  If $v \in X,$  i.e., if there is $c>0$ such that for
sufficiently large $|n|$ (where $n$ is even if $bc= Per^+$ or odd if
$bc= Per^-$)
\begin{equation} \label{a1}   c^{-1} |\beta^+_n
(z^*_n;v)| \leq |\beta^-_n (z^*_n;v)| \leq  c \,|\beta^+_n
(z^*_n;v)|,
\end{equation}
then there exists a Riesz basis in $L^2 ([0,\pi],\mathbb{C}^2)$ which
consists of eigenfunctions and at most finitely many associated
functions of the operator $L_{Per^\pm}(v).$
\end{Theorem}

{\em Remark.}  To avoid any confusion, let us emphasize that in
Theorem~\ref{thm0} two {\em independent} theorems are stacked
together: one for the case of periodic boundary conditions $Per^+$
(where we consider only even $n$), and another one for the case of
antiperiodic  boundary conditions $Per^-$ (where we consider only odd
$n$).

\begin{proof}

Let $N$ be chosen so large that the formula (\ref{g4}) in
Proposition~\ref{prop1} holds for $|n|>N$  (with a constant $c$
coming from (\ref{a1})), and the range $Ran (P_n)$ consists of
eigenfunctions only. In view of Corollary~\ref{cor3} such choice of
$N$ is possible. Moreover, we may assume without loss of generality
that $N$ is so large that the estimates (\ref{g2}) in
Lemma~\ref{lem21} holds for $|n|>N.$

We have the following two cases:

(a)  $\beta^-_n (z_n^*)= \beta^+_n (z_n^*)= 0;$

(b) $\beta^-_n (z_n^*)\neq 0, \; \beta^+_n (z_n^*)\neq  0.$

In Case (a) it follows from (\ref{g4}) that $\gamma_n = 0,$ so
$\lambda_n^*= n+z_n^*$ is a double eigenvalue of geometric
multiplicity two. In this case we choose eigenfunctions $f(n),
g(n)\in Ran (P_n) $ so that
\begin{equation} \label{a2}
\|f(n)\|=\|g(n)\|=1, \quad \langle f(n), g(n) \rangle =0.
\end{equation}

In Case (b) we have $\gamma_n \neq 0$ by Proposition~\ref{prop1}, so
$\lambda_n^- $ and $ \lambda_n^+$ are simple eigenvalues.  Now we
choose corresponding eigenvectors $f(n), g(n)\in Ran (P_n) $ so that
\begin{equation} \label{a2a}
\|f(n)\|=\|g(n)\|=1, \quad L_{Per^\pm}(v) f(n) =\lambda_n^- f(n),
\quad L_{Per^\pm}(v) g(n) =\lambda_n^+ g(n).
\end{equation}

In view of (\ref{r13}), to prove the theorem it is enough to show
that the system of eigenfunctions $\{f(n), g(n), \; |n|>N \}$ (where
$n$ is even for $bc= Per^+$ and odd for $bc= Per^-$) is a Riesz basis
in the space $H = \bigoplus_{|n|>N} Ran (P_n).$ In view of (\ref{r7})
in Lemma~\ref{lemr} it is enough to check that
$$
\sup_{|n|>N} |\langle f(n), g(n) \rangle  | < 1.
$$
Obviously, we need to consider only $n$ falling into Case (b). Let
$M$ be the set of all (even for $bc=Per^+$ or odd for $bc=Per^-$) $n$
such that $|n|>N $ and  (b) holds. Next we show that
\begin{equation} \label{a3}
\sup_{M} |\langle f(n), g(n) \rangle  | < 1.
\end{equation}

By Lemma \ref{lem21} the quotient $\eta_n (z)= \beta^-_n
(z)/\beta^+_n (z) $ is a well defined analytic function on a
neighborhood of the disc $K_n=\{z: \;|z-z^*_n| \leq \gamma_n \}.$
Moreover, in view of (\ref{g2}) and (\ref{a1}), we have
\begin{equation} \label{a4}
\frac{1}{4c}  \leq |\eta_n (z)| \leq 4c \quad \text{for} \quad n\in
M, \; z \in K_n.
\end{equation}
Since $\eta_n (z)$ does not vanish in $K_n ,$ there is an appropriate
branch $\text{Log}$ of $\log z $ (which depend on $n$) defined on a
neighborhood of $\eta_n (K_n). $  We set
$$
\text{Log} \, (\eta_n (z)) = \log |\eta_n (z)| + i \varphi_n (z);
$$
then
\begin{equation}
\label{a14} \eta_n (z)= \beta^-_n (z)/\beta^+_n (z) =|\eta_n (z)|
e^{i \varphi_n (z)}
\end{equation}
so the square root $\sqrt{\beta^-_n (z)/\beta^+_n (z)}$ is a well
defined as analytic function on a neighborhood of $K_n$ by
\begin{equation}
\label{a15} \sqrt{\beta^-_n (z)/\beta^+_n (z)}= \sqrt{|\eta_n
(z)|}e^{\frac{i}{2} \varphi_n (z)}.
\end{equation}

Now the basic equation (\ref{be}) splits into the following two
equations
\begin{eqnarray}
\label{a17} z=\zeta_n^+ (z):=  a(n,z) + \beta^+_n (z)
 \sqrt{\beta^-_n (z)/\beta^+_n (z)}, \\
\label{a18} z=\zeta_n^- (z):=  a(n,z) - \beta^+_n (z) \sqrt{\beta^-_n
(z)/\beta^+_n (z)}.
\end{eqnarray}
 For large enough $n,$ each of the equations (\ref{a17})
and (\ref{a18}) has exactly one root in the disc $D=\{z:\,
|z|<1/4\}.$ Indeed, in view of (\ref{1.37}),
$$
\sup_{|z|\leq 1/2} \left | d \zeta_n^{\pm}/dz \right | \to 0 \quad
\text{as} \quad n \to \infty.
$$
Therefore, for large enough $n$ each of the functions $\zeta_n^\pm $
is a contraction on the disc $K_n, $  which implies that each of the
equations (\ref{a17}) and (\ref{a18}) has at most one root in the
disc $K_n.$

On the other hand, by Lemma~\ref{loc} for large enough $n$ the basic
equation has two simple roots in $K_n, $ which implies that each of
the equations (\ref{a17}) and (\ref{a18}) has exactly one root in the
disc $K_n.$

For large enough $n,$ let $z_1 (n) $ (respectively $z_2 (n) $) be the
only root of the equation (\ref{a17}) (respectively (\ref{a18})) in
the  disc $D.$ Of course, we have either $z_1 (n) = \lambda_n^- -n,
\; z_2 (n) = \lambda_n^+ -n$ or
 $z_1 (n) = \lambda_n^+ -n, \; z_2 (n) = \lambda_n^- -n.$ Therefore,
\begin{equation}
\label{a24} |z_1 (n) - z_2 (n)| = \gamma_n = |\lambda_n^+
 -  \lambda_n^-|.
\end{equation}

We set
\begin{equation}
\label{c24} f^0(n)= P_n^0 f(n), \quad  g^0(n)= P_n^0 g(n).
\end{equation}
From (\ref{r12}) it follows that $\|P_n - P_n^0\| \to 0.$  Therefore,
$$
\|f(n) - f^0(n) \| = \|(P_n - P_n^0)f(n) \| \leq \|P_n - P_n^0\| \to
0
$$
and $\|g(n) -g^0 (n)\| \to 0, \; |\langle f(n)-f^0(n), g(n)-g^0 (n)
\rangle| \to 0.$  Since $ \|f (n)\|^2 =\|f^0 (n)\|^2+ \|f(n) - f^0(n)
\|^2 $ and $ \langle f(n), g(n) \rangle = \langle f^0(n), g^0(n)
\rangle + \langle f(n)-f^0(n), g(n)-g^0 (n) \rangle, $  we get
\begin{equation}
\label{a23} \|f^0(n)\|, \,\|g^0(n)\|\to 1,  \quad \limsup_{n\to
\infty} |\langle f(n), g(n) \rangle |=\limsup_{n\to \infty} |\langle
f^0(n), g^0(n) \rangle |.
\end{equation}

Then, by \cite[Lemma 21]{DM15} (see formula (2.4)),  $f^0 (n)$ is an
eigenvector of the matrix $\begin{pmatrix} \alpha_n (z_1) & \beta_n^+
(z_1)
\\ \beta_n^- (z_1) & \alpha_n (z_1)
\end{pmatrix}$ corresponding to its eigenvalue $z_1=z_1 (n), $ i.e.,
$$ \begin{pmatrix} \alpha_n (z_1)-z_1 &  \beta_n^+ (z_1) \\ \beta_n^- (z_1) &
\alpha_n (z_1)-z_1  \end{pmatrix} f^0 (n) = 0. $$ Therefore, $f^0(n)
$ is proportional to the vector $\left (1, \frac{z_1 - \alpha_n
(z_1)}{\beta_n^+ (z_1)} \right )^T.$ Taking into account (\ref{a14}),
(\ref{a15}) and (\ref{a17}) we obtain
\begin{equation}
\label{a25} f^0 (n) = \frac{\|f^0 (n)\|} {\sqrt{1+ |\eta_n (z_1)|}}
\begin{pmatrix} 1\\ \sqrt{|\eta_n (z_1)|}e^{\frac{i}{2}\varphi(z_1)}
\end{pmatrix}.
\end{equation}
In an analogous way, from (\ref{a14}), (\ref{a15}) and (\ref{a18}) it
follows
\begin{equation}
\label{a26} g^0 (n) = \frac{\|g^0 (n)\|} {\sqrt{1+ |\eta_n (z_2)|}}
\begin{pmatrix} 1\\ -\sqrt{|\eta_n (z_2)|}e^{\frac{i}{2}\varphi(z_2)}
\end{pmatrix}.
\end{equation}
Now,  (\ref{a25}) and (\ref{a26}) imply
\begin{equation}
\label{a27} \langle f^0 (n), g^0 (n) \rangle  =\|f^0 (n)\|\|g^0 (n)\|
\frac {1-\sqrt{|\eta_n (z_1)|}\sqrt{ |\eta_n (z_2)|} \, e^{i\psi_n }}
{\sqrt{1+ |\eta_n (z_1)|}\sqrt{1+ |\eta_n (z_2)|}},
\end{equation}
where
$$\psi_n = \frac{1}{2}[\varphi_n (z_1(n)) -\varphi_n (z_2(n)].
$$

Next we explain that
\begin{equation}
\label{a31} \psi_n \to 0 \quad \text{as} \;\; n \to \infty.
\end{equation}
Since $\varphi_n = Im \, \left ( \text{Log}
 \,\eta_n  \right )  $ we obtain, taking into account (\ref{a24}),
$$
|\varphi_n (z_1(n)) -\varphi_n (z_2(n)| \leq \sup_{[z_1,z_2]} \left |
\frac{d}{dz} \left (\text{Log} \,\eta_n \right )  \right | \cdot
\gamma_n,
$$
where $[z_1, z_2]$ denotes the segment with end points $z_1 = z_1
(n)$ and $z_2 = z_2 (n).$

By (\ref{1.37}) in Proposition \ref{bprop} and (\ref{g2}) in Lemma
\ref{lem21} we estimate
$$
\frac{d}{dz} \left (\text{Log} \,\eta_n \right )= \frac{1}{\beta^-_n
(z)} \frac{d\beta^-_n}{dz} (z) -\frac{1}{\beta^+_n (z)}
\frac{d\beta^+_n}{dz} (z), \quad z \in [z_1,z_2],
$$
as follows:
$$
\left | \frac{d}{dz} \left (\text{Log} \,\eta_n \right )  \right |
\leq \frac{\varepsilon_n}{|\beta^-_n (z_n^*)|}
+\frac{\varepsilon_n}{|\beta^+_n (z_n^*)|}
$$
where $\varepsilon_n= C \left ( \mathcal{E}_{|n|} (r)
+\frac{1}{\sqrt{|n|}} \right )  \to 0 \quad \text{as}\; n \to
\infty.$ Therefore, from (\ref{a1}) and (\ref{g4}) it follows
$$
|\varphi_n (z_1(n)) -\varphi_n (z_2(n)| \leq 4(1+c)\cdot
\varepsilon_n \to 0,
$$
i.e., (\ref{a31}) holds.

From (\ref{a27}) it follows
\begin{equation}
\label{a32} |\langle f^0 (n), g^0 (n) \rangle |^2=\|f^0 (n)\|^2 \|g^0
(n)\|^2 \cdot \Pi_n ,
\end{equation}
with
\begin{equation}
\label{a33} \Pi_n = \frac{1+|\eta_n (z_1)| |\eta_n
(z_2)|-2\sqrt{|\eta_n (z_1)||\eta_n (z_2)|} \cos \psi_n}{\left (1+
|\eta_n (z_1)| \right )\left (1+ |\eta_n (z_2)| \right )}.
\end{equation}

 If (\ref{a1}) holds, then  (\ref{a31}) implies  $\cos \psi_n
>0 $ for large enough $n,$  so taking into account that $\|f^0 (n)\|,
\|g^0 (n)\|\leq 1, $ we obtain by (\ref{a4})
$$ |\langle f^0 (n), g^0 (n) \rangle
|^2 \leq \Pi_n \leq
 \frac{1+|\eta_n (z_1)|
|\eta_n (z_2)| }{\left (1+ |\eta_n (z_1)|
 \right ) \left (1+ |\eta_n (z_2)|
 \right )} \leq \delta < 1
$$
with
$$
\delta= \sup \left \{\frac{1+xy}{(1+x)(1+y)}: \; \frac{1}{4c} \leq
x,y \leq 4c \right \} = \frac{1+16c^2}{(1+4c)^2}.
$$
 Now (\ref{a23}) implies that (\ref{a3}) holds, hence the
system of normalized eigenfunctions and associated functions is a
(Riesz) basis in $L^2 ([0,\pi]).$ The proof is complete.
\end{proof}

In fact, Theorem \ref{thm0} says that (\ref{a1}) is a sufficient
condition which guarantees

 (i) the system of root functions of $L_{Per^\pm} (v)$  is
  complete and has at most finitely many
  linearly independent associated functions;

 (ii) there exists a Riesz bases in $L^2 ([0,\pi], \mathbb{C}^2)$
which consists of root functions of the operator $L_{Per^\pm} (v).$

Besides the case $v \in X_t $ (see the next section for definition of
the class of potentials $X_t$)  it seems difficult to verify the
condition (\ref{a1}). Moreover, since the points $z_n^*$ are not
known in advance, in order to check (\ref{a1}) one has to consider
the values of the functions $\beta^\pm_n (z) $ for all $z$ close to
0.

In the next theorem we consider potentials $v$ such that for large
enough $|n|$
\begin{equation}
\label{c2}
 \beta_n^- (0)\neq 0, \quad \beta_n^+ (0)\neq 0
\end{equation}
and
\begin{equation}
\label{c3} \exists d>0 : \;\;  d^{-1}|\beta_n^\pm (0)| \leq
|\beta_n^\pm (z)| \leq d \, |\beta_n^\pm (0)| \quad   \forall z\in
K_n=\{z: |z-z_n^*|\leq \gamma_n \}
\end{equation}
(notice that $K_n $ consists of one point only if $\gamma_n =0 $).
Then (i) holds, and moreover, the condition (\ref{a1}) is necessary
and sufficient for existence of Riesz bases consisting of root
functions of the operator $L_{Per^\pm} (v).$

\begin{Theorem}
\label{thm00}  Suppose $v$ is a Dirac potential such that  (\ref{c2})
and (\ref{c3}) hold. Then

 (a) the system of root functions of $L_{Per^\pm} (v)$  is
  complete and has at most finitely many
  linearly independent associated functions;

  (b) if
\begin{equation}
\label{c4} 0 < a:=\liminf  \frac{|\beta_n^- (0)|}{|\beta_n^+ (0)|},
\qquad b:= \limsup  \frac{|\beta_n^- (0)|}{|\beta_n^+ (0)|}<\infty,
\end{equation}
where $n$ is even if $bc= Per^+$ or odd if $bc= Per^-,$ then there
exists a Riesz basis in $L^2 ([0,\pi], \mathbb{C}^2)$ which consists
of root functions of the operator $L_{Per^\pm} (v);$

(c) if (\ref{c4}) fails, then there is no basis in $L^2 ([0,\pi],
\mathbb{C}^2)$ consisting of root functions of the operator
$L_{Per^\pm} (v).$

\end{Theorem}

{\em Remark.}    Although the conditions (\ref{c2})--(\ref{c4}) look
too technical there is -- after \cite{DM15, DM10}  -- a well
elaborated technique to evaluate these parameters and check these
conditions. To compare with the case of Hill operators with
trigonometric polynomial coefficients -- see \cite{DM25a,DM25}.

\begin{proof}
By Proposition \ref{bprop}, the basic equation
\begin{equation}
\label{c5} (z-\alpha_n (z))^2 = \beta^+_n (z) \beta^-_n (z),
\end{equation}
has exactly two roots (counted with multiplicity) in the disc $D=\{z:
|z|< 1/4\}. $ Therefore, a number $\lambda = n +z $ with $ z \in D $
is a periodic or antiperiodic eigenvalue of algebraic multiplicity
two if and only if $z \in D$ satisfies the system of two equations
(\ref{c5}) and
\begin{equation}
\label{c6} 2(z-\alpha_n (z)) \frac{d}{dz} \left (z - \alpha_n (z)
\right ) =\frac{d}{dz} \left ( \beta^+_n (z) \beta^-_n (z)\right ).
\end{equation}

In view of \cite[Theorem 9]{DM20}, the system of root functions of
the operator $L_{Per^\pm} (v)$ is complete, so Part (a) of the
theorem will be proved if we show that there are at most finitely
many $n$  such that the system (\ref{c5}), (\ref{c6}) has a solution
$z \in D.$

Suppose $z^* \in D $ satisfies (\ref{c5}) and (\ref{c6}); then it
follows $z^* \in K_n.$  By (\ref{1.37}),  for each $z \in D $
\begin{equation}
\label{c7} \left |\frac{d\alpha_n}{dz}  (z) \right | \leq
\varepsilon_n, \quad  \left |\frac{d\beta^\pm_n}{dz}  (z)\right |
\leq \varepsilon_n \quad \text{with} \;\; \varepsilon_n \to 0
\;\;\text{as} \;\; |n| \to \infty.
\end{equation}
In view of (\ref{c7}), the equation (\ref{c6}) implies
$$
2\left | z^*-\alpha_n (z^*) \right | (1- \varepsilon_n) \leq
\varepsilon_n  \left (|\beta^+_n (z^*)|+|\beta^-_n (z^*)|   \right ).
$$
By (\ref{c5}),
$$
\left | z^*-\alpha_n (z^*) \right |= |\beta_n^+ (z^*) \beta_n^-
(z^*)|^{1/2},
$$
so it follows, in view of (\ref{c3}),
$$
2(1- \varepsilon_n)\leq \varepsilon_n \left ( \left | \frac{\beta^+_n
(z^*)}{\beta^-_n (z^*)} \right |^{1/2}  +\left | \frac{\beta^-_n
(z^*)}{\beta^+_n (z^*)} \right |^{1/2}  \right ) \leq 2 d
\varepsilon_n.
$$
Since $\varepsilon_n \to 0 $ as $|n| \to \infty, $  the latter
inequality holds for at most finitely many $n,$ which completes the
proof of (a).

If (\ref{c4}) holds, then by Theorem~\ref{thm0} there exists a Riesz
bases in $L^2 ([0,\pi], \mathbb{C}^2)$ which consists of root
functions of the operator $L_{Per^\pm} (v),$ i.e., (b) holds.

 Next, we show that  if (\ref{c4}) fails then there is no
bases in $L^2 ([0,\pi], \mathbb{C}^2)$ which consists of root
functions of the operator $L_{Per^\pm} (v).$

By (a) and Lemma~\ref{loc}, for large enough $|n|, $ say $|n|> N $
there are two simple (periodic for even $n$ and antiperiodic for odd
$n$) eigenvalues $\lambda^-_n $ and $\lambda^+_n $ close to $n.$ Let
us choose corresponding unit eigenfunctions $f(n) $ and $g(n),$
i.e.,
\begin{equation} \label{c10}
\|f(n)\|=\|g(n)\|=1, \quad L_{Per^\pm}(v) f(n) =\lambda_n^- f(n),
\quad L_{Per^\pm}(v) g(n) =\lambda_n^+ g(n).
\end{equation}
The same argument as in the proof of Theorem~\ref{thm0} shows that
there is a bases in $L^2 ([0,\pi], \mathbb{C}^2)$ which consists of
root functions of the operator $L_{Per^\pm} (v)$ if and only if
\begin{equation} \label{c11}
\sup \{ |\langle f(n), g(n) \rangle |: \; |n|>N \} < 1,
\end{equation}
where we consider even $n$ for periodic boundary conditions $bc=
Per^+ $ or odd $n$ for antiperiodic boundary conditions $bc= Per^-.
$

By Lemma (\ref{c3}) the quotient $\eta_n (z)= \beta^-_n
(z)/\beta^+_n (z) $ is a well defined analytic function on a
neighborhood of the disc $\overline{D}$ which does not vanishes on
$\overline{D}.$ Therefore, there is an appropriate branch (depending
on $n$) $\text{Log} $ of $\log z $ defined in a neighborhood of
$\eta_n (\overline{D}).$ We set
$$
\text{Log} (\eta_n (n)) = \log |\eta_n (z)| + i \varphi_n (z);
$$
then (\ref{a14})  holds.

Further we follow the proof of Theorem~\ref{thm0}, after formula
(\ref{a14}). With $f^0(n) $ and $g^0(n) $ given by (\ref{c24}) the
formulas (\ref{a23})--(\ref{a27}) and (\ref{a32}), (\ref{a33})
hold. In view of (\ref{c2}) and (\ref{c3}), if (\ref{c4}) fails then
\begin{equation} \label{c13}
\text{either} \;\; \liminf \left (\inf_{K_n} |\eta_n (z)|\right ) =
0  \quad \text{or} \;\; \limsup \left (\sup_{K_n} |\eta_n (z)|\right
)  = \infty.
\end{equation}
By (\ref{a32}), it follows $ \limsup \, \Pi_n =1, $ so  (\ref{a31})
and (\ref{a23}) imply
$$\limsup \{ |\langle f(n), g(n) \rangle |:\; |n|>N \}=1, $$
i.e., (\ref{c11}) fails. Therefore, if (\ref{c4}) fails there is no
bases in $L^2 ([0,\pi], \mathbb{C}^2)$  consisting of root functions
of the operator $L_{Per^\pm} (v), $  i.e., (c) holds. This completes
the proof.
\end{proof}

 \begin{Example}
If $a, b, A, B \in \mathbb{C}\setminus \{0\}$ and
 \begin{equation}
 \label{c15}
v=\begin{pmatrix} 0  & P  \\  Q  & 0   \end{pmatrix} \quad
\text{with} \quad
 P(x)  = a e^{ 2ix} + b e^{- 2ix}, \quad Q(x) = A e^{ 2ix} + B
e^{ -2ix},
\end{equation}
 then the system of root functions of $L_{Per^\pm}(v)$
 consists eventually of eigenfunctions.

 Moreover, for $bc=Per^-$ this system is a Riesz basis
in $L^2 ([0,\pi], \mathbb{C}^2)$ if  $|aA|=|bB|,$ and it is not a
basis if $|aA| \neq |bB|.$

For $bc=Per^+$ the system of root functions is a Riesz basis always.
\end{Example}

Let us mention that if $bc = Per^+$ then it is easy to see by
(\ref{1.22}), (\ref{23.34}) and (\ref{23.37}) that $\beta_n^\pm (z)=
0 $ whenever defined, so the claim follows by Theorem~\ref{thm0}.

If $bc = Per^-,$ then the result follows from Theorem~\ref{thm00} and
the asymptotics
 \begin{equation}
 \label{c16}
\beta_n^+ (0) = A^{\frac{n+1}{2}}a^{\frac{n-1}{2}}4^{-n+1} \left
[\left ( \frac{n-1}{2}  \right )! \right ]^{-2} \left (1+
O(1/\sqrt{|n|} \right ),
 \end{equation}
\begin{equation}
 \label{c17}
\beta_n^- (0) = b^{\frac{n+1}{2}}B^{\frac{n-1}{2}}4^{-n+1} \left
[\left ( \frac{n-1}{2} \right )! \right ]^{-2} \left (1+
O(1/\sqrt{|n|} \right ).
 \end{equation}
 Proofs of (\ref{c16}), (\ref{c17}) and similar asymptotics,
 related to other trigonometric polynomial potentials and implying
 Riesz bases existence or non-existence, will be given elsewhere
 (see similar results for the Hill-Schr\"odinger operator in
 \cite{DM25a,DM25}).

\section{Density of finite zone potentials in the class $X_t$}

Consider the classes of Dirac potentials
\begin{equation}
\label{4.2}  X_t =\left \{v=
\begin{pmatrix} 0 & P
\\ Q & 0 \end{pmatrix}, \quad Q (x) = t \overline{P(x)}, \quad
 P,Q \in L^2 ([0,\pi]) \right \}, \quad t \in \mathbb{R} \setminus \{0\}.
\end{equation}
If $t=1$ we get the class $X_1$ of symmetric Dirac potentials (which
generate self-adjoint Dirac operators), and $X_{-1} $ is the class of
skew-symmetric Dirac potentials. In this section we show that
\begin{equation}
\label{4.3}  X_t \subset X \quad  \forall \,t \in \mathbb{R}
\setminus \{0\},
\end{equation}
and prove that  finite-zone $X_t$-potentials are dense in $X_t$ for
real $t\neq 0.$

\begin{Lemma}
\label{lem4.3} (a)  The Dirac operators (\ref{001}) with potentials $
v =
\begin{pmatrix} 0  & P  \\   Q  &  0 \end{pmatrix}  $ and $ v_c =
\begin{pmatrix} 0  & cP  \\  \frac{1}{c} Q  &  0 \end{pmatrix},  $
$c \in \mathbb{C}\setminus \{0\},$  are similar. Therefore, $Sp \,
(L_{Per^\pm} (v_c))$ does not depend on $c.$

(b) $ (L_{Per^\pm} (v))^* = L_{Per^\pm} (v^*), \quad v^* =
\begin{pmatrix} 0  & \overline{Q}  \\  \overline{P}  &  0 \end{pmatrix}.  $

(c) If $t\neq 0$ is real and $v\in X_t$ then $v^* = v_t,$  so
\begin{equation}
\label{4.9} Sp \, [(L_{Per^{\pm}}(v))^*] = Sp \, (L_{Per{^\pm}} (v)).
\end{equation}

\end{Lemma}

\begin{proof}
Let $C= \begin{pmatrix} c  & 0  \\  0  &  1
\end{pmatrix};  $  then
$C^{-1} = \begin{pmatrix} 1/c  & 0  \\  0   &  1
\end{pmatrix},  $ and  we have
$$ C L(v) C^{-1} = i CJDC^{-1} + CvC^{-1} = iJD + v_c= L(v_c). $$

Moreover, if $G= \begin{pmatrix}  g_1 \\ g_2
\end{pmatrix} $ satisfies periodic (or antiperiodic)
boundary conditions, $CG =  \begin{pmatrix}  cg_1 \\ g_2
\end{pmatrix} $  satisfies the same boundary conditions, and vice
versa. Thus, the operators $L_{Per^\pm} (v_c)$ and $L_{Per^\pm} (v)$
are similar.

Part (b) is standard.

Since  $$ v^* =\begin{pmatrix} 0 & \overline{Q}
\\ \overline{P} & 0 \end{pmatrix} =
\begin{pmatrix} 0 & tP
\\ \frac{1}{t} Q & 0 \end{pmatrix} = v_t,
$$  (\ref{4.9}) follows from Part (a).

\end{proof}

If $v \in X_t$  and $c\neq 0 $ is real, then
\begin{equation}
\label{4.21}  v_c = CvC^{-1} =
\begin{pmatrix} 0 & cP
\\ \frac{t}{c} \overline{P}  & 0 \end{pmatrix} \in X_{t/c^2}.
\end{equation}
This observation and \ref{4.9} lead to the following. specification
of Lemma~\ref{loc} for potentials in the classes $X_t. $

\begin{Lemma}
\label{lem4.4} (a)  If $v \in X_t $ with $t>0, $ then $L_{Per^\pm}
(v)$ is similar to a self-adjoint operator, so $ Sp  \left
(L_{Per^\pm} (v) \right ) \subset \mathbb{R}. $

(b)  If $v \in X_t $ with $t<0, $ then there is an $N=N(v) $ such
that for $|n|>N$ either

(i) $\lambda_n^-$ and $\lambda_n^+$ are simple eigenvalues and
$\overline{\lambda_n^+} =\lambda_n^-, \;  Im \, \lambda^\pm_n \neq 0$

or
 (ii) $\lambda_n^+ =\lambda_n^-$ is a real eigenvalue of
algebraic and geometric multiplicity 2.
\end{Lemma}

\begin{proof}
In view of Lemma \ref{lem4.3} and (\ref{4.21}), considered with $c=
\sqrt{|t|}, $ in case (a) the operator $L_{Per^\pm} (v)$ is similar
to a self-adjoint operator $L_{Per^\pm} (v_1)$ with $v_1 \in H_1.$

The same argument shows that in case (b) we need to consider only the
skew-symmetric case $t=-1.$ By Lemma \ref{loc}, there is an $N=N(v) $
such that for $|n|>N$ the disc $D_n =\{z: \; |z-n|<1/4\}$ contains
exactly two (counted with algebraic multiplicity) periodic (for even
$n$) or antiperiodic (for odd $n$) eigenvalues of the operator
$L_{Per^\pm}.$ By (\ref{4.9}) in Lemma~\ref{lem4.3}, if $\lambda \in
D_n $ with $  Im \, \lambda \neq 0 $ is an eigenvalue of $L_{Per^\pm}
$  then $\overline{\lambda} \in D_n $ is also an eigenvalue of
$L_{Per^\pm}$ and  $\overline{\lambda} \neq \lambda, $ so $\lambda $
and $\overline{\lambda}$ are simple, i.e., (i) holds.

Suppose $\lambda \in D_n $ is a real eigenvalue. We are going to show
that $\lambda $ is of geometric multiplicity two, i.e., (ii) holds.

Let $\begin{pmatrix} w_1 \\w_2
\end{pmatrix}$  be a corresponding (nonzero) eigenvector, i.e.,
$$
L \begin{pmatrix} w_1 \\w_2
\end{pmatrix}
= \lambda  L \begin{pmatrix} w_1 \\w_2
\end{pmatrix}.
$$
Passing to conjugates we obtain
$$
L \begin{pmatrix} \overline{w_2} \\ -\overline{w_1}
\end{pmatrix}
= \lambda  L \begin{pmatrix} \overline{w_2} \\ -\overline{w_1}
\end{pmatrix},
$$
i.e., $\begin{pmatrix} \overline{w_2} \\ -\overline{w_1}
\end{pmatrix}$ is also an eigenvector corresponding to the same eigenvalue $\lambda.$
But $\left \langle    \begin{pmatrix} w_1 \\w_2
\end{pmatrix},
\begin{pmatrix} \overline{w_2} \\ -\overline{w_1}
\end{pmatrix}    \right  \rangle =0, $
so these vector-functions are orthogonal, and therefore, linearly
independent. This completes the proof of Lemma~\ref{lem4.4}.
\end{proof}

\begin{Proposition}
\label{prop6} Suppose $v \in X_t $ with $ t\neq 0$ real. Then there
is $N=N(v)$ such that
\begin{equation}
\label{4.24} z_n^* = \frac{1}{2} (\lambda_n^- + \lambda_n^+) - n
\quad \text{is real for} \quad |n|>N.
\end{equation}
Moreover, for every real $ t \neq 0$
\begin{equation}
\label{4.25} \overline{\beta^+_n (z_n^*, v)}=t \cdot \beta^-_n
(z_n^*, v),
\end{equation}
which implies $v \in X, $ i.e.,
\begin{equation}
\label{4.26} X_t \subset X.
\end{equation}
\end{Proposition}

\begin{proof} Suppose $v \in X_t $ with $ t\neq 0$ real.
Lemma \ref{lem4.4} implies (\ref{4.24}) immediately. In view of
(\ref{1.22}) and (\ref{4.24}), it follows from Part (b) of
Lemma~\ref{lem2}, formula (\ref{1.20}), that (\ref{4.25}) holds. In
view of (\ref{g1}) we obtain $v \in X, $ which completes the proof.
\end{proof}

In view of Theorem \ref{thm0} and (\ref{4.26}) we have
\begin{Corollary}
If $v \in X_t $ then there is a Riesz basis in $L^2 ([0,\pi],
\mathbb{C}^2)$ which consists of eigenfunctions and at most finitely
many associated functions of the operator $L_{Per^\pm} (v).$
\end{Corollary}

 In view of Proposition~\ref{prop6},
(\ref{4.24}) and (\ref{4.25}), for sufficiently large $N$ the
nonlinear operators (compare with \cite[(3.52)--(3.54)]{DM15})
\begin{equation}
\label{4.31} A_N (v) = v + \Phi_N (v), \quad \Phi_n (v) =
\begin{pmatrix} 0 & \Phi_N^{12}  \\ \Phi^{21} (v) &  0,
\end{pmatrix}
\end{equation}
where
\begin{equation}
\label{4.32} \Phi_N^{12}= \sum_{|n|>N} (\beta_n^- (v, z_n^*) -
p(-n))e^{-2inx} \quad \text{and} \quad \Phi_N^{21}= \sum_{|n|>N}
(\beta_n^+ (v, z_n^*) - q(n))e^{2inx},
\end{equation}
are well-defined, and  $$v \in X_t \quad \Rightarrow \quad \Phi_N
(v), \, A_N (v) \in X_t $$ as well. Therefore, all constructions and
proofs of \cite[Section 3.4]{DM15} for symmetric (self-adjoint)
potentials become valid for any $X_t$-potential.

Moreover, in \cite[Theorem 70]{DM15} the density of finite-zone
potentials is first proved for general Dirac potentials, and then the
$A_N$-invariance of the space symmetric potentials is used (see
Remark 56 therein) to derive that the symmetric finite-zone
potentials are dense in any weighted space of symmetric potentials.
So, without any need to repeat or reproduce hard analysis we can
claim the following analog of \cite[Theorem 70]{DM15}.

\begin{Theorem}
\label{thm2} If $\Omega $ is a submultiplicative weight and  $X_t
(\Omega ) = X_t \cap H_D (\Omega ) $ is the corresponding Sobolev
space of $X_t$-Dirac potentials, then the finite-zone $X_t
$-potentials are dense in $X_t (\Omega).$
\end{Theorem}

For skew-symmetric potentials  (i.e., when $t= -1$) Theorem~\ref{thm2}
is proved in   \cite{KST} (see Corollary~1.1 there).
See more comments about   \cite{KST} in Appendix.

\section{Appendix: remarks on the paper  \cite{KST} }

Presumably,  T. Kappeler, F. Serier and P. Topalov  (the authors of \cite{KST})  have not noticed
that $X_{skew-sym}$ potentials are invariant for {\em all} operators
$\Phi_N, \, A_N, $ etc. in \cite{DM15}, and they rewrote all
technical constructions, lemma by lemma, inequality by inequality from
\cite{DM6} or \cite{DM15} to justify  analogs of 
\cite[Theorems~58, 70] {DM15}
for skew-symmetric potentials. 
 But such copying is done in \cite{KST} without
specifying which lemmas and inequalities are rewritten and without
explaining that the entire architecture of \cite{DM6, DM15} is
reproduced.

 Moreover, the main results of  \cite{KST} follow immediately from Lemma 49 and
Theorem~68 in \cite{DM15}  (or, from Lemmas 48, 49 and Proposition 57 in \cite{DM15})
but this fact is not mentioned.
Appendix is aimed to cover these  gaps  in \cite{KST}, at
least partially. \vspace{3mm}

1. First, let us recall \cite[Theorem 68]{DM15}.

{\bf Theorem 68 in \cite{DM15}.}  
{\em If $$
 L = L^0  + v(x),
\quad L^0 = i
\begin{pmatrix}
 1  &  0 \\
0  &  -1
\end{pmatrix} \frac{d}{dx},
 \quad v(x)=
\begin{pmatrix}
 0  &  P(x) \\
Q(x)  &  0
\end{pmatrix}
$$ is a  periodic Dirac operator  with $L^2 $--potential (i.e.,
$P$ and $Q$ are periodic  $L^2 ([0,1])$--functions), then, for
$|n|> n_0 (v),  n \in \mathbb{Z},$ the operator $L$ has, in the
disk of center $n $  and radius $r = 1/4, $ exactly two (counted
with their multiplicity) periodic (for even $n$), or
anti--periodic (for odd $n$)   eigenvalues $\lambda^+_n $ and $
\lambda^-_n, $ and one Dirichlet eigenvalue $\mu_n. $

Let
\begin{equation}
\label{44.13} \Delta_n = |\lambda^+_n - \lambda^-_n | +
|\lambda^+_n - \mu_n|, \quad |n| > n_0;
\end{equation}
then, for each sub--multiplicative weight $\Omega, $
\begin{equation}
\label{44.14} v \in H(\Omega )  \; \Rightarrow \;  (\Delta_n)  \in
\ell^2 (\Omega ).
\end{equation}

Conversely, if  $\Omega = (\Omega (n))_{n\in \mathbb{Z}} $ is a
sub--multiplicative weight  such that
\begin{equation}
\label{44.15}  \frac{\log \Omega (n)}{n} \searrow 0 \quad
\text{as} \quad n \to \infty,
\end{equation}
then
\begin{equation}
\label{44.16}  (\Delta_n) \in \ell^2 (\Omega) \;  \Rightarrow  \;
v \in H(\Omega ).
\end{equation}

If $\Omega $ is a sub--multiplicative weight of exponential type,
i.e.,
\begin{equation}
\label{44.17} \lim_{n\to \infty}  \frac{\log \Omega (n)}{n} >0
\end{equation}
then}
\begin{equation}
\label{44.18} (\Delta_n) \in \ell^2 (\Omega) \; \Rightarrow   \;
\exists  \varepsilon >0: \; v \in H(e^{\varepsilon |n|} ).
\end{equation}

The main results in \cite{KST}  -- see Theorems 1.2 and 1.3 there --
follow from \cite[Theorem 68]{DM15} because
for skew-symmetric potentials  $v$
\begin{equation}
\label{ap1}
\Delta_n \leq K \gamma_n, \quad |n| \geq N_1 (v),
\end{equation}
where $K$ is an absolute constant.  Indeed, by
\cite[Theorem 66]{DM15}  we have for any $L^2$ potential $v$
\begin{equation}
\label{ap2}
 \frac{1}{144} \left ( |\beta^-_n (z_n^*)| +|\beta^+_n
(z_n^*)| \right ) \leq\Delta_n  \leq
54 \left ( |\beta^-_n (z_n^*)| +|\beta^+_n (z_n^*)| \right ),
\quad |n| \geq N(v).
\end{equation}
On the other hand,  since
\begin{equation}
\label{ap3}
|\beta^-_n (z_n^*,v)| = |\beta^+_n
(z_n^*,v)|   \quad \text{for skew-symmetric} \; \; v,
\end{equation}
 \cite[Lemma 49]{DM15} implies easily
\begin{equation}
\label{ap5}  \gamma_n \geq  D
 \left ( |\beta^-_n (z_n^*)| +|\beta^+_n
(z_n^*)| \right ) ,
\end{equation}
where $D$  is an absolute constant. 
 Therefore, (\ref{ap1}) holds
for skew-symmetric
potentials. \vspace{3mm}

2. The authors of \cite{KST} did not say anything about the relation between
\cite[Theorem 68]{DM15} and their main results but they explained 
(see p. 2087 in \cite{KST})  that they wrote a
big portion in their 40 page long paper to fill a gap in the paper
\cite{DM15}. They write: "The proof of Lemma 36 in [5] has a gap on
p. 710 as Lemma 32 in [5] cannot be applied to the expression
$\Sigma_4 (n) $ given by (2.117) of [5]. However it turns out that
the method developed in [4] can be applied."

Maybe in \cite{DM15} not everything is explained letter-by-letter
(which  is common in mathematical research papers)
 but without any extra effort the expression $\Sigma_4 (n) $
 given by (2.117) of \cite{DM15} can be estimated by
 Lemma~32. Indeed, we have
 \begin{equation}
 \label{2.117}
\Sigma_4 (n)= \langle  \hat{V}\hat{D} \hat{T}^2 (1-\hat{T}^2)^{-1}
\hat{T} \hat{V} e_n^1, e_n^2 \rangle = \sum_{i,j \neq n}
\frac{r(n+i)r(n+j)}{|n-i||n-j|} h^{21}_{ij} (n)
\end{equation}
with
 \begin{equation}
 \label{2.117a}
h=   \hat{T}^2 (1-\hat{T}^2)^{-1}
 \hat{V}.
\end{equation}
Therefore, if $\Sigma_4 (n)$ is written as (\ref{2.117}) with $h$
given by (\ref{2.117a}) then Lemma 32 immediately yields the
inequality (2.122) on page 710 of \cite{DM15}.
\vspace{3mm}

3.  Lemma 49 in \cite{DM15}  is essentially \cite[Lemma 12]{DM5}
(its formulation
and proof are the same for Dirac and
 Hill-Schr\"odinger operators).   It plays a crucial role in
getting estimates  of $\gamma_n $ from below in terms
of $  \left ( |\beta^-_n (z_n^*)| +|\beta^+_n
(z_n^*)| \right )$  (see \cite[Section 3.2]{DM15})
which leads to two-sided estimates
of $\gamma_n $  (see \cite[Theorem 50]{DM15}).
In Section 7 of  \cite{KST}  (titled "lower bound for $\gamma_n"$),
Lemma~7.1   is almost identical to
 Lemma 49 in \cite{DM15}  or  \cite[Lemma 12]{DM5}
but the authors did not give any credit to the papers
\cite{DM5,DM15}.

Let us mention also that with (\ref{ap5}) proven 
for skew-symmetric potentials 
the main results of \cite{KST} 
follow from  
\cite[Lemma 48 and Proposition 57]{DM15}  in the same way as \cite[Theorem 58]{DM15}
(self-adjoint case) -- namely,   if $(\gamma_n)\in \ell^2 (\Omega), $
then by
(\ref{ap5}) and Lemma 48  there is a slowly growing weight $\Omega_1 $ 
such that $A_N (v) \in H(\Omega \cdot \Omega_1),$ 
so \cite[Proposition~57]{DM15}   implies  $v \in H(\Omega \cdot \Omega_1) \subset H(\Omega).$  
\vspace{3mm}

4.  Theorems~1.1 and 4.1 in \cite{KST}  claim  (\ref{44.14}),
respectively, for skew-symmetric and
arbitrary potentials $v.$   Of course, Theorem~1.1 is a
partial case of Theorem~4.1
but the proof in both cases is the same.
The authors of \cite{KST} write on p. 2075:
"To make the paper self-contained we include
for convenience of the reader   a proof of Theorem 1.1."
Although it is written in the introduction
that, in the generality stated, Theorem 1.1 (or Theorem 4.1)
is proven in \cite{DM6,DM15}, 
it is not said that the proof of Theorem 1.1 in \cite{KST} is copied from there.
In particular:

(i)  Lemma 2.2 on p. 2083 in \cite{KST}
(which is crucial for the proof of Lemmas 4.1 and 4.2,
and therefore for the whole paper)
reproduces Lemma~2 in \cite{DM6} and its proof -- see pp. 144--146
there -- but no credit  is given to   \cite{DM6}.

(ii)  Proposition 4.1 in \cite{KST}  (which also provides the crucial a priori  estimate
in the proof of Theorems~1.2 and 1.3 \cite{KST})  reproduces Lemma 36 in
\cite{DM15}. 

(iii) Corollary 4.1, p. 2096 in \cite{KST} and its inequality is essentially the same as   
pp. 163--164 in \cite{DM6}, in particular Inequality (4.18). \vspace{3mm}

5.  Proposition 5.1, p. 2100 in \cite{KST}, reproduces  -- with all essential steps in the proof -- Lemma 55, p. 729 in \cite{DM15}.  One semantic remark, however, should be made. 
Instead of straightforward application of Banach-Cacciopolli contraction principle --
as it has been done in \cite{Mit00} and explained and further used in \cite{GK03}
 -- the authors of \cite{KST} prefer to talk about 
Implicit Function Theorem for analytic diffeomorphisms. 
  But an abstract approach -- Fixed Point Theorem or Implicit Function Theorem -- works only if some hard analysis 
is done, and the authors of \cite{KST} copy such analysis from \cite{DM6,DM15}. 
 \vspace{3mm}

6.    Proposition 6.2,  p. 2103 in \cite{KST}, reproduces  
\cite[Proposition~57]{DM15}  but this fact is not mentioned in \cite{KST}. 
\vspace{3mm}

7. This comparative analysis could go on and on but at the end we have
to mention ONE   case when \cite{KST} give a (funny) credit to \cite{DM15},
see p. 2108, lines 1--4 from the bottom: 

"To prove Theorem 1.2 and Theorem 1.3, we want to apply Theorem~8.1.  
The following lemma which can be found in \cite[Lemma 48]{DM15} allows 
to get rid of the weight function $w$ in Theorem~8.1. 

{\bf Lemma 9.1.} {\em If $z=(z_k)_{k\in\mathbb{Z}},$  then there exists an unbounded slowly increasing weight $w=(w(k))_{k\in\mathbb{Z}}$  such that} $z \in h^w. $"

Of course,  the above statement is a version of a well known Calculus exercise: if a series of positive numbers 
$\sum c_k $ converges then there is a sequence $D_k \to \infty $ such that  
the series $\sum c_k D_k $ converges as well.

This is the only lemma for which the authors of \cite{KST} 
give credit to \cite{DM15} although they rewrote tens of other lemmas and inequalities
and their proofs without saying  anything about the origin. 
But quite paradoxically this "credit" is very revealing.
It shows that \cite{KST} borrow the entire structure of the proof from \cite{DM6,DM15}
so even such a minor item as an elementary lemma (Lemma 48 in \cite{DM15}) could not be omitted;
without that brick the proofs of their main results would not exist.

\end{document}